\documentclass[10pt]{article}
\usepackage[left=2.5cm,right=2.5cm,top=2cm,bottom=2cm,
            footskip=1cm]{geometry}
\usepackage{graphicx,float} 
\usepackage{amsmath,amsthm,amssymb,amsfonts,geometry,mathtools}
\usepackage[hidelinks]{hyperref}
\usepackage[titles]{tocloft}
\usepackage[square,numbers]{natbib}
\usepackage[absolute]{textpos}
\usepackage{xcolor}
\usepackage{bm}
\usepackage{graphicx}
\usepackage{bbm}
\usepackage{comment}
\usepackage{etoolbox}
\usepackage{enumitem}
\numberwithin{equation}{section}
\AtEndEnvironment{equation}{\noindent}
\AtEndEnvironment{align}{\noindent}
\AtEndEnvironment{gather}{\noindent}
\DeclareMathOperator{\E}{\mathbb{E}}
\DeclareMathOperator{\R}{\mathbb{R}}

\newtheorem{theorem}{Theorem}[section]
\newtheorem{lemma}[theorem]{Lemma}
\newtheorem{proposition}[theorem]{Proposition}
\newtheorem{definition}[theorem]{Definition}
\newtheorem{remark}[theorem]{Remark}
\newenvironment{keywords}{
  \noindent\textit{Keywords: }\itshape}{\par}

\DeclarePairedDelimiter{\floor}{\lfloor}{\rfloor}
\begin{document}
 
\title{Stochastic numerical approximation for nonlinear Fokker-Planck equations with singular kernels}
\author{N. Cazacu\thanks{CMAP, INRIA, École polytechnique, Institut Polytechnique de Paris, 91120 Palaiseau, France.}}
\date{}
\maketitle
\renewcommand\thefootnote{}
\footnote{
{\href{mailto:nicoleta.cazacu@polytechnique.edu}{Email:\texttt{
    nicoleta.cazacu@polytechnique.edu}}}}
\footnote{We acknowledge the support  the SDAIM project ANR-22-CE40-0015 from the French National Research Agency.}
\vspace{-0.75cm}
\begin{abstract}
    This paper studies the convergence rate of the Euler-Maruyama scheme 
    for systems of interacting particles used to approximate solutions of nonlinear 
    Fokker-Planck equations with singular interaction kernels, such as
    the Keller-Segel model. We derive 
    explicit error estimates in the large-particle limit for two objects:  the empirical 
    measure of the interacting particle system and 
    the density distribution of a single particle. 
    Specifically, under certain assumptions on the interaction kernel and 
    initial conditions, we show that the convergence rate of both objects
     towards solutions of the corresponding
    nonlinear Fokker-Planck equation depends polynomially on \( N \)  
    (the number of particles) and  on
    \( h \) (the discretization step).  The analysis shows
    that the scheme converges despite 
    singularities in the drift term. 
    To the best of our knowledge,
    there are no existing results in the literature of such 
    kind for the singular kernels considered in this work.  
\end{abstract}

%------------------------------------------------------------
%------------------------------------------------------------
\section{Introduction}

In this work, we obtain convergence rates of 
an Euler-Maruyama (EM) scheme
used to approximate solutions of nonlinear Fokker-Planck
partial differential equations (PDEs) of the form 
\begin{equation}\label{PDE}
\begin{cases}
    \partial_t u(t,x) = \Delta u(t,x) - 
    \nabla \cdot (u(t,x) K \ast u(t,x)), \quad 
    t>0, x \in \mathbb{R}^d,\\
    u(0,x) = u_0(x),
\end{cases}
\end{equation}
where $K$ is a singular kernel satisfying 
certain assumptions (see 
conditions \hyperlink{assumption:AK}{\boldmath{$(A_K)$}} below). 
The singular kernels considered in this work
are of interest in various domains, such as 
biology (e.g. the chemotaxis model \cite{P15}), physics 
(e.g. Coulomb potentials for electrostatic and gravitational forces)
and specifically in
fluid dynamics (e.g. Biot-Savart kernel for $2d$ Navier-Stokes equation \cite{CM93}).
In particular, this allows for 
attractive and singular interactions, such as 
the Keller-Segel interaction kernel (see \cite{KS70}).

Now consider a measure dependent stochastic differential equation (SDE),
also called  McKean-Vlasov SDE, of the form
\begin{equation}\label{SDE}
    \begin{cases}
    dX_t = K \ast u_t(X_t)dt + \sqrt{2} dW_t, \quad t>0,\\
    \mathcal{L}(X_t) = u_t,\,\mathcal{L}(X_0) = u_0,
    \end{cases}
\end{equation}
where $W$ is a standard Brownian motion and $\mathcal{L}$ denotes the law 
of a random variable. It is well-known that the law of 
the process $(X_t)_{t \geq 0}$  defined by \eqref{SDE}
satisfies 
the PDE \eqref{PDE} in a distributional sense, and under some not so 
restrictive assumptions
on $u_0$.\\

In order to approximate numerically the solution of the PDE \eqref{PDE} via 
the McKean-Vlasov 
SDE \eqref{SDE}, one might initially 
consider using an EM scheme for the SDE
\eqref{SDE}. 
However, this approach faces a key difficulty: at each time step 
of the EM scheme,
one would need to compute the density of the solution, 
which is not directly available.
This challenge can be overcome using a particle-based method inspired 
by the propagation of
 chaos theory, see \citet{Sznitman}. The core idea 
is to approximate the density \( u \) by the empirical measure of
a system of particles. Consider $N$ diffusion processes 
$(X^{i,N})_{i \in \{1, ...,N \}}$ driven by independent Brownian motions 
and interacting through their empirical measure and the kernel $K$. 
Let \( h > 0 \) denote the discretization step, 
and define the discrete time grid 
\(\tau_s^h := \lfloor \frac{s}{h} \rfloor h\). 
A discrete-time representation of the previously 
described system of $N$
diffusion processes would lead to the following EM scheme:
\begin{equation}\label{SIP}
    X_t^{i,N,h} = X_0^{i,N,h} + \int_0^t K \ast \mu_{\tau_s^h}^{N,h}
    (X_{\tau_s^h}^{i,N,h})\,ds 
    + \sqrt{2} W_t^{i}, \quad 
    t \geq 0, \, 1 \leq i \leq N,
\end{equation}
where $\mu^{N,h}_{t} : = 
\frac{1}{N} \sum_{i=1}^N \delta_{X_t^{i,N,h}}$ is the empirical measure
 and $(W^i)_{1 \leq i \leq N}$ is a 
family of independent 
standard Brownian motions on a probability space ($\Omega, \mathcal{F}, 
(\mathcal{F}_t)_{t \geq 0}, \mathbb{P}$).

However, when formally letting \( h \) 
tend to $0$ in \eqref{SIP}, 
we obtain a particle system whose well-posedness is not 
always guaranteed when the interaction kernel \( K \) exhibits 
singularities.
To circumvent this issue,
we introduce a mollification of the interaction kernel $K$ as considered 
by \citet{oelschlager1985law}, \color{black}
as well as a control cut-off function to ensure numerical
stability of the scheme. \color{black} 
Namely, let $(V^N)_{N \in \mathbb{N}^*}$ be a sequence of mollifiers defined in 
Subsection \ref{sec2.1}.
The system \eqref{SIP} is replaced by the following
regularized system of particles: 
\begin{equation}\label{EMSDE}
    X_t^{i,N,h} = X_0^{i,N,h} + \int_0^t 
    F_A(K \ast V^N \ast \mu_{\tau_s^h}^{N,h}(X_{\tau_s^h}^{i,N,h}))\,ds 
    + \sqrt{2} W_t^{i}, \quad 
    t \geq 0, \, 1 \leq i \leq N,
\end{equation}
where
$F_A$ is a smooth function (see Subsection \ref{sec2.1})
that 
cuts off extreme values of 
the drift, which is consistent from a numerical point 
of view as it reflects the limitations of computer storage 
for large numerical values. The presence of $F_A$
is also justified from a theoretical perspective, as 
it ensures a uniformly bounded drift in $N$
 for the particle system \eqref{EMSDE}, which 
 is consistent with the boundedness of the drift term \( K \ast u_t \) 
when considering the limit \( N \to \infty \)
\color{black}(the range of the cut-off will be calibrated 
along the bound $\|K \ast
u_t\|_{L^\infty}$). \color{black}

Recently, a quantitative convergence result was established in
\citet{ORT} for the non-discretized version of \eqref{EMSDE} towards the 
McKean-Vlasov SDE \eqref{SDE}. In this work, we tackle the 
question of quantitative convergence of the
EM discretized system of particles \eqref{EMSDE}. 
The strategy is to combine the 
techniques developed by \citet{ORT} (for 
Theorem \ref{thm.1}) with techniques in
 \citet{jourdain} (for Theorem \ref{thm.2}) and establish
convergence rates for the EM scheme applied to the system of 
interacting particles under singular kernel assumptions. As
in \cite{ORT},
we use a semigroup approach to decompose the difference between 
the regularized empirical measure of the system of particles \eqref{EMSDE} 
and the density of the McKean-Vlasov \eqref{SDE}. Then combining
a fine analysis of each term of the difference including a stochastic 
convolution term, we establish, for any $m\geq1$, the following 
convergence rate for the mollified empirical measure 
\(\tilde{\mu}_t^{N,h} := V^N \ast \mu_t^{N,h}\) 
towards \(u_t\), the solution of the PDE \eqref{PDE}:

\[
\mathbb{E} \left[ \sup_{s \in [0,t]}\left\| \tilde{\mu}_s^{N,h} - u_s 
\right\|_{L^1\cap L^r}^m \right]^{\frac{1}{m}}
\leq C \mathbb{E} \left[ \left\| \tilde{\mu}_0^{N,h} - u_0 
\right\|_{L^1\cap L^r}^m \right]^{\frac{1}{m}} 
+ C N^{-v_1} + C N^{v_2} h^{v_3}.
\]
where $r$ and \(v_1, v_2, v_3 > 0\) depend on the kernel 
assumptions, as detailed in Theorem \ref{thm.1}.

Then 
we study the convergence of
the density of the $i$-th particle in the EM scheme towards the density
of the McKean-Vlasov particle,
with an approach inspired by 
\cite{jourdain}. This involves expanding
 the densities around the heat kernel. 
 Each term of the expansion is then carefully bounded to manage 
 the effects of irregularities and discretization errors. For any 
 \(t \in [0, T]\), any \(1 < p < \frac{d}{d-1} \) and almost any \(x \in \mathbb{R}^d\), 
we obtain the following rate of convergence for \(u^{i,N,h}_t\), 
the density of the \(i\)-th particle of the EM scheme, to \(u_t\), 
the density of the McKean-Vlasov SDE:

\[
|u^{i,N,h}_t(x) - u_t(x)| 
\leq C \left( \mathbb{E} \left[ \left\| \tilde{\mu}_0^{N,h} - u_0 \right\|_{L^1\cap 
L^r}^{\bar{p}} \right]^{\frac{1}{\bar{p}}} + N^{-v_1} + N^{v_2} h^{v_3} \right)
\left( \int_{\mathbb{R}^d} g_{c/p}(t, x-z) u_0(dz) \right)^{1/p}.
\]
Here \(g_c(t, \cdot)\) is the Gaussian density 
on \(\mathbb{R}^d\) centered at zero with covariance matrix 
\(ct \mathbb{I}_d\), and the other parameters are further 
specified in Section 2. 

In practice, if \textcolor{black}{the initial error} $\| \tilde{\mu}_0^{N,h} - u_0 \|_{L^1\cap L^r}$ 
\textcolor{black}{converges to $0$ fast enough}, the convergence rate of the empirical measure and of the 
density of a particle towards the solution of the PDE are of order 
\(O(N^{-v_1} + N^{v_2} h^{v_3})\). For bounded Lipschitz kernels, 
this corresponds to a convergence rate of order 
\(O(N^{-\left(\frac{1}{d+2}\right)^-} + h^{\frac{1}{2}})\), while 
for Coulomb type kernels 
$K \varpropto \nabla |x|^{2-d}$, when 
$d \geq 2$,
the rate is \(O(N^{-\left(\frac{1}{2(d+1)}\right)^-}+N^{
    \left(\frac{d}{2(d+1)}\right)^+} h^{\frac{1}{2}^-})\).
Further discussions are provided in Section 4.

The challenge in deriving both convergence rates 
lies in ensuring that the bounds explicitly depend on \(N\) 
(the number of particles which goes to $\infty$) 
and \(h\) (the time discretization step which tends to $0$). 
This dependence must be accurately quantified to reflect the 
trade-off between particle-based approximations and 
discretization-induced errors, particularly under the constraints 
imposed by the singularity of the interaction kernel.\\

\textbf{Literature:}
The use of stochastic particle methods to 
approximate solutions of PDEs can be traced back to \citet{Chorin}, 
who applied them to reaction-diffusion equations. This approach was later 
extended by Bossy and Talay in \cite{BT95} for one-dimensional 
nonlinear
PDEs with
bounded Lipschitz drifts and 
bounded $C^1_b$ uniformly positive diffusion, 
and also in \citet{BT94}
for one-dimensional 
PDEs of Burgers type 
(for the one-dimensional Burgers model, an interesting 
comparison with classical numerical analysis techniques was provided by 
\citet{BFP97}). 
Later on, the result was improved and 
generalized 
by \citet{B00} for one-dimensional viscous scalar conservation laws, and 
a strong convergence rate for the empirical 
cumulative distribution function
was obtained.
The first results on the optimal convergence rates of
 the Euler scheme for interacting particle systems were obtained by
\citet{KO97}. Building on the Malliavin calculus techniques
 they used, \citet{AK02} improved the convergence rate in \cite{BT95}.
We also refer the reader to \citet{B05} for a review of 
convergence results and rates of convergence for mean-field particle approximations, 
particularly for Lipschitz coefficients and viscous scalar conservation laws.
For a more recent survey, see \citet{CD22}, where stochastic particle methods are 
discussed.
Recently, new results have been obtained in general dimensions and 
for drift $b(t,x,\mu)$ and diffusion $\sigma(t,x,\mu)$ coefficients
of McKean-Vlasov SDEs having different regularity assumptions.
For instance, \citet{BH19} and \citet{L24} assume Lipschitz 
continuity of the coefficients in the Wasserstein distance with
respect to $\mu$ and Hölder or Lipschitz continuity in $x$ and $t$, and
\citet{FS25} consider Hölder-type regularity assumptions for both 
the drift and the diffusion coefficients.
Other scenarios involve linear or super-linear growth conditions, as in \citet{Z19} and \citet{CLL24}, or a combination of local Lipschitz continuity and uniform linear growth, as in \citet{LMSWY23}. In some cases, specific hypotheses allow for discontinuous coefficients in one-dimensional problems. For example, \citet{BJ20} consider 
\( b(t,x,\mu) = \lambda(\mu(-\infty,x]) \), where \(\lambda\) is Lipschitz continuous, while \citet{LRS22} explore the case 
\( b(t,x,\mu) = b_1(x) + b_2(x,\mu) \), with \( b_1 \) Lipschitz continuous on \( (-\infty,0) \) and \( (0,+\infty) \), and \( b_2 \) Lipschitz continuous in both variables.
Across these studies, strong or weak convergence results are derived, often
with explicit convergence rates for discretized particle systems towards 
the McKean-Vlasov SDE. In this work, the main novelty is to consider drifts of the form
$b(t,x,\mu) := K \ast \mu(x)$ where $K$ is singular and 
$\mu$ is in a certain class of measures. In this case, the drift 
is not Lipschitz continuous in the Wasserstein norm for a generic measure $\mu$, nor is it 
Hölder continuous in $x$ or of linear growth in $x$, see Remark \ref{rem1}.\\

In general, there are few results of convergence of the EM 
scheme when applied to non-interacting SDEs with irregular coefficients 
and this is a dynamic area of research.
Several works, such as \citet{NT18,DG20,YCS21,NS21,JIP23}
 and references therein, study weak and strong convergence rates 
 of the EM scheme under various drift irregularity assumptions. 
 These include cases of possible discontinuity and exponential or 
 polynomial growth.
When the drift is a singular \textcolor{black}{time-space}
 function of \textcolor{black}{$L^q([0,T],L^p(\R^d))$-type, 
 satisfying some Krylov-Röckner \cite{Krylov-Rockner} condition,}
\citet{ll22} obtained a strong convergence rate while
\citet{jourdain}  obtained convergence rates for the density of the EM scheme.
 Nonetheless, in our case, it is not possible to
analyze the system of interacting particles as a single 
high-dimensional SDE (of dimension \(N \times d\)) 
 because the dependence in the dimension
 of the rate of convergence  
 leads to rate constants that grow unbounded and tend to infinity as
  \(N \to \infty\). \\

Finally, the idea of using a sequence of mollifiers 
in the context of singular drifts was introduced
by \citet{oelschlager1985law}, then used
in \citet{meleard1987propagation} and 
adopted in many other works since then, e.g. in \citet{JM98,ORT,ORT24,HJM24}.\\

\textbf{Plan of the paper:} In Section \ref{sec2}, we present the
 problem's setting along with the notations and assumptions 
(originally 
 introduced in \cite{ORT}). This is 
 followed by the statement and a discussion of the main results. 
 Section \ref{sec3} contains detailed proof of the main results
 establishing the two convergence rates. Finally, Section \ref{sec4}
  illustrates applications for different examples of 
  singular kernels and provides a discussion on the 
  corresponding convergence rates in each case.

%------------------------------------------------------------
%------------------------------------------------------------

\section{Results}\label{sec2}

In this section, we start by
introducing the main assumptions and notations. Then, we state
our main results on the convergence of the empirical measure and the 
density of a given particle towards
 the solution of the Fokker-Planck equation.

%------------------------------------------------------------
%------------------------------------------------------------
 
\subsection{Notations and assumptions}\label{sec2.1}
We begin by introducing some notations.

\begin{itemize}
    \item We denote by $\bar{p}$ the 
    Hölder conjugate of $p \in [1, +\infty]$
        : $\frac{1}{p} + \frac{1}{\bar{p}} = 1$.
    \item The unit ball of $\mathbb{R}^d$ is denoted by $B_1$.
    \item For all $ p\geq 1$, we denote by $\| \cdot \|_{L^p}$ the
    $L^p$ norm on $\mathbb{R}^d$.
    \item For all $ p, q \geq 1$, we define
     $\| \cdot \|_{L^p \cap L^q} :=
    \| \cdot \|_{L^p} + \| \cdot \|_{L^q}$.
    \item For all $T >0$, we define  $\| \cdot \|_{T, L^p}:= 
    \sup_{t \in [0, T]} \| \cdot \|_{L^p}$.
    \item The Hölder semi-norm with parameter $\zeta$ is denoted by
     $[f]_{\zeta} := \sup_{x \neq y} \frac{|f(x) - f(y)|}{|x - y|^\zeta}$.
    \item The Gaussian density on $\mathbb{R}^d$, centered and with covariance matrix 
    $ctI_d$, is denoted by $g_c(t, x)$.
    \item $C$ represents a generic constant that may change throughout the article.
\end{itemize}

Throughout this work, our kernel $K$ is assumed to satisfy the following conditions:

\begin{itemize}[label={}, leftmargin=40pt]
    \item[\boldmath{$(A_K):$}]~\hypertarget{assumption:AK}{}
    \begin{itemize}
        \item[($A^K_i$)] $K \in L^p(B_1)$, for some $p \in [1, +\infty]$;
        \item[($A^K_{ii}$)] $K \in L^q(B_1^c)$, for some $q \in [1, +\infty]$;
        \item[($A^K_{iii}$)] There exists $r \geq \max(\bar{p}, 
        \bar{q})$, $\zeta \in (0, 1]$, and $C > 0$ such 
        that for all $f \in L^1 \cap L^r$, we have
        \[
            [ K \ast f ]_{\zeta} \leq C\|f\|_{L^1 \cap L^r}.
        \]
    \end{itemize}
\end{itemize}

Note that, by Hölder's inequality, ($A^K_i$) and ($A^K_{ii}$) imply the following 
useful inequality:
\[\|K \ast f \|_{L^{\infty}} \leq 
C \|f\|_{L^1 \cap L^r},\]
where $C$ depends on $K$ and on the dimension $d$.

\begin{remark}\label{rem1}
As mentioned in the Introduction, for this type of numerical approximation,
an example of structural assumption on the drift $b(t,x,\mu)$ (see
\cite{BH19,L24}) is \textcolor{black}{Lipschitz continuous with respect to 
the
measure argument in a $L^p$-Wasserstein space (for $0< p<+\infty$)}, namely
\[
|b(t,x,\mu) - b(t,x,\nu)| \le C\, \mathbb{W}_p(\mu,\nu),
\]
where $\mathbb{W}_p$ denotes the $L^p$-Wasserstein distance.

In the present work, however, the drift has the specific convolutional
structure $b(t,x,\mu) = K * \mu(x)$ with a singular kernel $K$.
In this setting, a regularity condition expressed in terms of a
generic Wasserstein distance is not well suited.
Indeed, if $\mu=\delta_x$ and $\nu=\delta_y$, then
$\mathbb{W}_1(\mu,\nu)=|x-y|$, while
$$
|b(t,z,\mu)-b(t,z,\nu)|
= |K(z-x)-K(z-y)|,
$$
which may be unbounded when $K$ has singularities. \color{black}We may
also observe that a condition of the form $|K \ast \mu(z) - K\ast \nu(z)| 
\leq C \mathbb{W}_p(\mu,\nu)$ amounts, taking $p \in (0,1]$ for 
simplicity, to have, for any coupling $\pi \in \Pi(\mu,\nu)$ between 
$\mu$ and $\nu$,
$$| \int K(z-x) - K(z-y) d\pi(x,y)| \leq C \int_{\pi \in \Pi(\mu,\nu)}
 |x-y|^p d\pi(x,y).$$
Such estimate requiring to show the $p$-Hölder continuity of $K$ over minimal couplings
of $\mu$ and $\nu$ is seemingly unsuitable for the class of singular kernels 
considered in our 
framework. \color{black}

A more appropriate framework is provided by the regularity properties
of the marginal laws of the McKean-Vlasov SDE \eqref{SDE}.
Under the assumptions considered here, these laws admit densities
belonging to $L^1\cap L^r$.
Consequently, assumption $(A^K_{iii})$ expresses a Hölder-type
regularity of the convolution operator in terms of the
$L^1\cap L^r$ norm.
\end{remark}

\begin{remark}\label{remlplq}
    Among singular kernels of interest in the literature 
    are those that satisfy the Krylov-Röckner condition :
    \color{black}
        $$ K \in L^q([0,T],L^p(\mathbb{R}^d))\quad
         \text{with}\quad
        p,q \in [1, +\infty], \quad \frac{d}{p} + 
        \frac{2}{q} < 1 \quad \text{and} \quad T<\infty,$$\color{black}
    as this allows one to prove the existence of a 
    unique strong solution for the associated McKean-Vlasov SDE 
    \cite{Krylov-Rockner}, but also 
    propagation of chaos for the associated system of particles
    \cite{tomasevic23}. For $K$
    independent of time, 
    the Krylov-Röckner condition is equivalent here to having 
    $K \in L^p$ for some $p \in (d, +\infty]$. This 
    is a stronger integrability condition than the one
    required in ($A^K$).

    On the other hand, a generic kernel $K \in L^p, \,p>d$, does not necessarily
    have enough structure 
    to ensure 
    that it satisfies condition $(A^K_{iii})$. However, one can 
    consider specific examples of $L^p$ kernels for $p > d$ 
    exhibiting singularities and verify that they 
    satisfy $(A^K_{iii})$. For instance, consider the kernel 
    \begin {equation}\label{eq:example}
        K(x) = \frac{x}{|x|^{\alpha}} \chi(x),\,\, x \in \mathbb{R}^d,
    \end{equation}
    where $\alpha \in (1, 2)$ and $\chi$ is a smooth 
    function equal to $1$ on $B_1$ and $0$ outside $B_2$. 
    In this case, $K \in L^p$ for $p 
    \in (1, \frac{d}{\alpha - 1})$. Therefore, this singular 
    kernel satisfies the Krylov-Röckner condition. Notice that in a more 
    singular regime, when $\alpha \geq 2 $, such 
    kernels no longer satisfy the
    Krylov-Röckner condition but can still fall within the scope of our results, 
    see Sections 4 
and 5 in \cite{ORT} and the following
Section 4 for more details.

    Let us now verify that the kernel $K$, defined by \
    eqref{eq:example}, satisfies $(A^K_{iii})$. 
    Take $z > d$ and $r \geq \max(\bar{p}, z)$. Then, 
    for any $f \in L^1 \cap L^r$, using 
    Young's inequality with $\frac{1}{p} + \frac{1}{m} = 
    1 + \frac{1}{z}$ and that $1 \leq m \leq \bar{p} \leq r$, we have
    \[
        \|K \ast f\|_{L^z} \leq C \|K\|_{L^p} \|f\|_{L^m} 
        \leq C \|f\|_{L^1 \cap L^r}.
    \]

    Moreover, one can easily verify 
    that the matrix-valued kernel $\nabla K $ is in $L^1$ for $d \geq 2$. 
    Knowing that $1 \leq z \leq r$, we obtain
    \[
        \|\nabla K \ast f\|_{L^z} \leq
         C \|\nabla K\|_{L^1} \|f\|_{L^z}
          \leq C \|f\|_{L^1 \cap L^r}.
    \]

    This implies that 
    $K \ast f \in W^{1,z}$ for 
    some $z > d$. By Morrey's inequality (see Thm. 9.12 in \cite{brezis11}), 
    it follows that $K \ast f$ is Hölder continuous 
    with exponent $\zeta := 1 - \frac{d}{z}$. Thus, $K$ 
    satisfies condition $(A^K_{iii})$.
\end{remark}

Let $A > 0$. In the following, we denote by 
    $F_A : x \in \mathbb{R}^d \to F_A(x) = (F_A(x)_1, 
\cdots, F_A(x)_d) \in \mathbb{R}^d$
a Lipschitz continuous and bounded function verifying

\[
\forall i \in \{1, \dots, d\}, \quad F_A(x)_i =
\begin{cases} 
    x_i & \text{if } |x_i| \leq A, \\
    A  & \text{if } x_i > A+1, \\
    -A  & \text{if } x_i < -A-1.
\end{cases}
\]

Let $V : \mathbb{R}^d \to \mathbb{R}_+$ be a smooth 
probability density function, and suppose that $V$ has a 
compact support. For any $x \in \mathbb{R}^d$ and $N \in \mathbb{N}^{*}$,
 define
\[
V^N(x) := N^{d\alpha} V(N^\alpha x),
\]
for some $\alpha \in [0, 1]$. In the sequel,
 $\alpha$ will be restricted to an interval 
$(0, \alpha_0)$, see Assumption \hyperlink{assumption:Aalpha}{\boldmath{$(A_\alpha)$}} below.\\

Let $h >0$, $N \in \mathbb{N}^*$ and consider the
 following 
 EM scheme: 
\begin{equation}\label{eds_euler}
    \begin{cases}
        X_t^{i,N,h} = X_0^{i,N,h} + \displaystyle \int_0^t 
        F_A\left(\frac{1}{N}\sum_{k=1}^N K\ast V^N(X_{\tau_s^h}^{i,N,h}
        - X_{\tau_s^h}^{k,N,h})\right)ds 
        + \sqrt{2} W_t^{i}, \quad 
        t \geq 0, \, 1 \leq i \leq N,\\
        X_0^{i,N,h} = X_0^{i,N}, \, 1 \leq i \leq N, \quad 
        \text{independent of} \, \, \{W^i, 1 \leq i \leq N\},
    \end{cases}
\end{equation}
where $\{(W^i_t)_{t \in [0,T]}, i \in \mathbb{N}^*\}$ is a 
family of independent Brownian motions valued in 
$\mathbb{R}^d$, defined on a filtered probability 
space $(\Omega, \mathcal{F}, (\mathcal{F}_t)_{t \geq 0}, 
\mathbb{P})$ and $\tau_s^h := \lfloor \frac{s}{h} \rfloor h$.\\

For $t \geq 0$, denote the  empirical measure on $
\mathbb{R}^d$ of the $N$ discretized SDEs \eqref{eds_euler} at time $t$ by
\[
\mu^{N,h}_{t} := \frac{1}{N} \sum_{i=1}^N \delta_{X_{t}^{i,N,h}}
\]
and the mollified empirical measure by 
\[\tilde{\mu}_{t}^{N,h} := V^N \ast \mu_{t}^{N,h}.\]

Throughout, it will be assumed that the parameters $(r,p,q,\zeta)$ 
are given 
in \hyperlink{assumption:AK}{\boldmath{$(A_K)$}} and are
such that $r \geq \max(\bar{p}, \bar{q})$. The restriction with respect to the parameter
$\alpha$ and on the initial conditions 
are given by the following assumptions:\\

\noindent {\boldmath{$(A_\alpha)$}}~\hypertarget{assumption:Aalpha}{} The 
parameter $\alpha$ satisfies
\[0 < \alpha < \alpha_0 := \frac{1}{d + 2d\left(\frac{1}{2} - 
\frac{1}{r}\right) \vee 0}.\]

\noindent \boldmath{$(A_0)$} \unboldmath{}~\hypertarget{assumption:A}{} 
For all $m \geq 1$,
\[
\sup_{1 \leq i \leq N, N \in \mathbb{N}^*} \mathbb{E}
|X^{i,N}_0|^m < \infty \quad \text{and} \quad 
\sup_{N \in \mathbb{N}^*} \mathbb{E} \left[\left\|{\mu}_0^{N,h} 
\ast V^N\right\|_{L^r}^m \right] < \infty.
\]
We recall the following useful properties of the Gaussian density:

\begin{equation}\label{eq:heat_kernel}
    \forall f \in L^p, \forall t \in (0, T],  \exists C > 0,\quad
    \|\nabla g_2(t,\cdot)\ast f\|_{L^p} \leq \frac{C}{\sqrt{t}} \|f\|_{L^p},
\end{equation}

\begin{equation}\label{eq:gaussian_density}   
    \forall c > 2, \exists C > 0, \forall t \in (0, T], \forall
    x \in \mathbb{R}^d, \quad
  |\nabla g_2(t, x)| \leq \frac{C}{\sqrt{t}} g_c(t,x).
\end{equation}

In this work, the solutions of the Fokker-Planck equation 
\eqref{PDE} will be understood in the sense of the following definition.
\begin{definition}\label{def}
    Given $K$ satisfying $({A_i^K})$ and $({A_{ii}^K})$, $u_0 \in L^1 
    \cap L^r$ with $r \geq \max(\bar{p}, 
    \bar{q})$ and $T > 0$, a function $u$ on $[0, T]
     \times \mathbb{R}^d$ is said to be a mild solution
     to \eqref{PDE} on $[0, T]$ if:  
    \begin{itemize}
        \item[(i)] $u \in C([0, T]; L^1 \cap L^r)$;
        \item[(ii)] $u$ satisfies the integral equation
    \end{itemize}
   \begin{equation}
        u_t = g_2(t,\cdot) \ast u_0 - \int_0^t \nabla \cdot
         \left(g_2(t-s,\cdot)\ast \big(u_s \, K \ast u_s\big)\right) ds, \quad 0 \leq t \leq T. 
    \end{equation}
\end{definition}

In \cite{ORT}, it is shown that there exists $T_{\max} > 0$ such that
the Fokker-Planck
equation \eqref{PDE} admits a unique mild solution in the sense of
Definition \ref{def}
 on $[0, T]$, for any $T < T_{\max}$. In all the following results, we denote by 
\( T \in (0, T_{\max}) \) a fixed time horizon 
{for which the mild 
solution in
Definition \ref{def} exists and is unique.
On the interval $[0, T]$, we also have
wellposedness (in the weak sense) of the McKean-Vlasov SDE \eqref{SDE},
 where $\mathcal{L}(X_t)$ 
 \textcolor{black}{(or more precisely its Lebesgue density function)}
 is the mild solution to \eqref{PDE}.}\\

In all the following, the cut-off $A$ of the function $F_A$ used in the EM 
scheme \eqref{eds_euler} is chosen to satisfy the following condition:\\

\noindent \boldmath{$(A_{F_A})$}\unboldmath{}~\hypertarget{assumption:A-cut-off}{}
The parameter $A$ satisfies
\[
A \geq  \sup_{t \in [0, T]} \|K \ast u_t\|_{L^\infty},
\]

\begin{remark}
    In practice, it is possible to choose $A$ in terms
    of the inputs $K$, $T$ and $u_0$. Indeed,
    by \hyperlink{assumption:AK}{$(A^K_{iii})$}, 
    $$ \|K \ast u_t\|_{L^\infty} \leq C_{K,d} \|u_t\|_{L^1 \cap L^r},$$
    where $C_{K,d}$ can be made explicit in terms of $K$ and $d$. 

    \begin{itemize}

    \item When the kernel is bounded, $C_{K,d} = \|K\|_{L^\infty}$ and $r = 1$, hence it is enough to choose
    $A \geq \|K\|_{L^\infty}$.
    
    \item For Krylov-Röckner type kernels discussed in Remark \ref{remlplq}, 
    $K \in L^p, p >d$ and 
    $\|K\ast u_t\|_{L^\infty} \leq \|K\|_{L^p} \|u_t\|_{L^{\bar{p}}}$.
     Using the mild formulation and the fact that 
     $\forall t \in (0, T], \exists C > 0,
    \|\nabla g_2(t,\cdot)\ast f\|_{L^{\bar{p}}}
     \leq C t^{-\frac{1}{2}-\frac{d}{2p}} \|f\|_{L^1}$, we obtain 
 $$\|u\|_{t, L^{\bar{p}}}
    \leq \|u_0\|_{L^{\bar{p}}} + C \|K\|_{L^p} \int_0^t 
    (t-s)^{-\frac{1}{2}-\frac{d}{2p}} 
    \|u\|_{s, L^{\bar{p}}}\, ds$$
    And using Grönwall's lemma (see for example
    Lemma 6.19 in \cite{DN02}), \[ \|u\|_{T, L^{\bar{p}}} 
    \leq \|u_0\|_{L^{\bar{p}}} E_{\frac{1}{2}-\frac{d}{2p}}(C 
    \|K\|_{L^p}T^{\frac{1}{2}-\frac{d}{2p}}), \] 
    where $E_{\alpha}$ is the Mittag-Leffler function of order 
    $\alpha$. Hence it is possible to explicitly
    lower bound $A$ depending 
    on
    only $T, \|K\|_{L^p}, d, p$ and $\|u_0\|_{L^{\bar{p}}}$.

    \item For general kernels satisfying \hyperlink{assumption:AK}{$(A_K)$}, using 
    the mild formulation and the properties of the heat kernel, it 
    is immediate to obtain for $t \leq T$,
    $$ \|u\|_{t, L^1 \cap L^r}
    \leq \|u_0\|_{L^1 \cap L^r} + C \sqrt{t} \|u\|_{t, L^1 \cap L^r}^2.$$
    Then, for $4 C \sqrt{T}\|u_0\|_{L^1 \cap L^r} < 1$, 
    we have $$\|u\|_{t, L^1 \cap L^r} \leq \frac{1 - 
    \sqrt{1 - 4C\sqrt{T}\|u_0\|_{L^1 \cap L^r}}}{4C\sqrt{T}}.$$

   We remark that the obtained bound imposes a 
   restriction on either the time horizon \(T\) or 
   the size of the initial data \(\|u_0\|_{L^1 \cap L^r}\). However, 
   this provides only a preliminary method
   and is not expected to yield the optimal choice of \(A\). For a given 
   kernel \(K\), a more precise analysis at the PDE level, exploiting the
   regularity properties of \(u_t\), should allow one to obtain sharper bounds for 
   \(\|K \ast u_t\|_{L^\infty}\), and consequently for \(A\).
    \end{itemize}

\end{remark}

%------------------------------------------------------------
%------------------------------------------------------------

\subsection{Main results}

The first main result is the following claim,
 whose proof is 
detailed in Section \ref{secthm1}. The goal 
is to prove the convergence of the mollified empirical measure for the EM discretized 
system of particles towards the solution of the PDE \eqref{PDE} and 
to provide a convergence rate.

\begin{theorem}\label{thm.1}
    Under the assumptions
    \hyperlink{assumption:A}{\boldmath{$(A_0)$}},
    \hyperlink{assumption:AK}{\boldmath{$(A_K)$}},
    \hyperlink{assumption:Aalpha}{\boldmath{$(A_\alpha)$}} and
    \hyperlink{assumption:A-cut-off}{\boldmath{$(A_{F_A})$}},
     for any $\epsilon > 0$ and any $m \geq 1$, there exists a 
    constant $C > 0$ such that for any $N \in \mathbb{N}^{\ast}$, for any $h>0$,

    \[
    \left\| \left\| \tilde{\mu}^{N,h} - u \right\|_{T,L^1\cap L^r} \right\|_{L^m(\Omega)}
    \leq C \left(\left\| \left\| \tilde{\mu}_0^{N,h} - u_0
     \right\|_{L^1\cap L^r} 
     \right\|_{L^m(\Omega)}
    + N^{-\rho+\epsilon} + N^{\frac{d \alpha}{\bar{r}}}
    h^{\frac{\zeta}{2}}\right).
    \]
    where $\rho = \min(\alpha \zeta,\frac{1}{2}(1-\alpha(d + \chi_r)))$
    and $\chi_r = \max(0, d(1-2/r))$.
\end{theorem}
The convergence rate requires a polynomial compromise between
the number of particles $N$ which tends to infinity and the
time step $h$ which tends to zero. See Section \ref{sec4}
for explicit convergence rates for different examples of singular kernels.
\\

Now $\forall i \in \{1,...,N\}$, $\forall t \in [0,T]$, denote
the density of $X_t^{i,N,h}$ by  $u^{i,N,h}_t$.
 The second main result is the following claim, whose proof is
detailed in Section \ref{secthm2} and relies on the previous analysis.

\begin{theorem}\label{thm.2}
 Assume 
 \hyperlink{assumption:A}{\boldmath{$(A_0)$}},
    \hyperlink{assumption:AK}{\boldmath{$(A_K)$}},
    \hyperlink{assumption:Aalpha}{\boldmath{$(A_\alpha)$}},
    \hyperlink{assumption:A-cut-off}{\boldmath{$(A_{F_A})$}}
    and that 
    $\forall i \in \{1,...,N\},\,\mathcal{L}(X_0^{i,N,h}) = u_0$. Then, 
 for all 
 $ 1 < p < \frac{d}{d-1}$ there exists a constant 
 $C > 0$ such that for any $N \in \mathbb{N}^{\ast}$, for any $h>0$,
 for almost any $x \in \mathbb{R}^d$, for 
 any $t \in [0,T]$, and for any $i \in \{1,...,N\}$,
\[
|u^{i,N,h}_t(x) - u_t(x)| 
\leq C \left( \left\| \left\| \tilde{\mu}_0^{N,h} - u_0
\right\|_{L^1\cap L^r} 
\right\|_{L^{\bar{p}}(\Omega)}+ N^{-\rho + \epsilon} + 
N^{\frac{d \alpha}{\bar{r}}} h^{\frac{\zeta}{2}} \right) 
\left(\int_{\mathbb{R}^d}
g_{c/p}(t, x-y)  u_0(y)\,dy\right)^{1/p},
\]
where $\rho$ 
is defined as in Theorem \ref{thm.1}.
\end{theorem}

We obtain the same convergence rate as in Theorem \ref{thm.1}, multiplied 
by the integral of a Gaussian density. 
This result implies the
quantification of the weak error for the EM scheme, in the sense that 
it provides a convergence rate for 
$\max_{1 \leq i \leq N}
|\mathbb{E}[\phi(X^{i,N,h}_t)] - \mathbb{E}[{\phi(X_t)}]|$ for a 
certain class of test functions $\phi$ which can even be Dirac masses.
Note that when $u_0$ is
 a Dirac mass in $x_0$, which is not possible in the context of 
 the singular 
 kernels we consider but is possible for $L^q-L^{\rho}$ kernels
 in \cite{jourdain}, we obtain the term $g_c(t, x-x_0)$ in the right-hand side of 
the inequality, 
which is also present in the convergence rate 
of the EM scheme in \cite{jourdain}.

\begin{remark}\label{initial-data}
    In practice, it is necessary to verify 
    that the initial condition satisfies \hyperlink{assumption:A}{\boldmath{$(A_0)$}} 
    and explicitly derive the 
    convergence rate of $$\left\| \left\| \tilde{\mu}_0^{N,h} - u_0
    \right\|_{L^1\cap L^r} 
    \right\|_{L^{m}(\Omega)}.$$ A first natural assumption to facilitate 
    this derivation is that the particles are initially independent and 
    identically distributed according to \(u_0\). Following the 
    computations in \cite{ORT24}, and assuming that \(u_0\)
    has a density $f$ with compact support and in \(L^r\), we have,
     for $m \geq r \geq 2 $, 
        \begin{align}
\mathbb{E} \left[\left\|{\mu}_0^{N,h}
        \ast V^N\right\|_{L^r}^m \right] \nonumber
        &\leq C( \mathbb{E} \left[\left\|{\mu}_0^{N,h}
        \ast V^N - \mathbb{E}[{\mu}_0^{N,h}
        \ast V^N]\right\|_{L^r}^m \right] + 
        \|V^N \ast f\|_{L^r}^m) \\ \nonumber
        & \leq C( \mathbb{E} \left[\left\|{\mu}_0^{N,h}
        \ast V^N - \mathbb{E}[{\mu}_0^{N,h}
        \ast V^N]\right\|_{L^r}^m \right] + 
         \| f\|_{L^r}^m) \nonumber
        \end{align}
        \color{black}
        Since $\int \mu_0^{N,h}\ast V^N(x)dx = 1$ and using
        Jensen's inequality,
        we have
        \begin{align*}&\mathbb{E} \left[\left\|{\mu}_0^{N,h}
        \ast V^N - \mathbb{E}[{\mu}_0^{N,h}
        \ast V^N]\right\|_{L^r}^m \right] \\
        &= \E \left( \int 
        |\mu_0^{N,h}\ast V^N(x) - \E[\mu_0^{N,h}\ast V^N(x)]|^{r-1}
        |\mu_0^{N,h}\ast V^N(x) - \E[\mu_0^{N,h}\ast V^N(x)]|
         dx\right)^{\frac{m}{r}} \\
        & \leq 2 \E \int |\mu_0^{N,h}\ast V^N(x) - 
        \E[\mu_0^{N,h}\ast V^N(x)]|^{\frac{(r-1)m}{r}}
        |\mu_0^{N,h}\ast V^N(x)- 
        \E[\mu_0^{N,h}\ast V^N(x)]| dx \\
        &= C\,
        \mathbb{E}\int_{\mathbb{R}^d} 
        \left|\frac{1}{N}\sum_{i=1}^N
        V^N(x-X_{0}^{i,N,h})-V^N\ast 
        f(x)\right|^ndx,
        \end{align*}
        where $n := 1 + \frac{m}{r}$. 
        Noticing that 
        $\sum_{i=1}^N \left(V^N(x-X_{0}^{i,N,h}) -
         V^N\ast f(x)\right)$
        is a sum of iid centered random variables, 
        and applying Marcinkiewicz-Zygmund moment 
        inequality (see e.g. Lemma 2 in \cite{M38})
        followed by 
        $\left(\sum_{i=1}^N |a_i|^2\right)^{n/2} 
        \leq N^{n/2-1} \sum_{i=1}^N |a_i|^{n}$,
        there exists $C > 0$ such that\color{black}
     \begin{align}
        \int_{\mathbb{R}^d}
       \nonumber
        & \mathbb{E}\left[ \left|\frac{1}{N}\sum_{i=1}^N
        V^N(x-X_{0}^{i,N,h})-V^N\ast f(x)\right|^n\right] dx 
        \leq
        C N^{-\frac{n}{2}} (n-1)!! \int_{\mathbb{R}^d}
         \mathbb{E}\left[ |V^N(x-X_{0}^{1,N,h})
        -V^N\ast f(x)|^n\right] \, dx \\ \nonumber
        &\leq C N^{-\frac{n}{2}} \int_{\mathbb{R}^d}
        \int_{\mathbb{R}^d}  \left|
        \int_{\mathbb{R}^d} (V^N(x-y) - V^N(x-z)) f(z) dz \right|^n f(y) \,dy \, dx \\ \nonumber
        &\leq C N^{-\frac{n}{2}} 
        \int_{(\mathbb{R}^d)^3} |V^N(x-y) - V^N(x-z)|^n f(z) f(y) \, dy \, dz \, dx\\ \nonumber
        & \leq C N^{-\frac{n}{2}}
          \int_{(\mathbb{R}^d)^3} |V^N(x-y)|^n f(z) f(y) \, dy \, dz \, dx
        \\ \label{unif_bound}
        &\leq C N^{-\frac{n}{2} + d \alpha (n-1)},
     \end{align}
      and \eqref{unif_bound} is uniformly bounded in $N$ if 
      $\alpha \leq \frac{1}{2d}$. The same conclusion holds for $1 \leq m < r$ 
      using Hölder's inequality :
      $$ \left\| \left\|{\mu}_0^{N,h}
    \ast V^N  - \mathbb{E}[{\mu}_0^{N,h}
        \ast V^N]\right\|_{L^r}\right\|_{L^m(\Omega)} \leq \left\| \left\|{\mu}_0^{N,h}
    \ast V^N  - \mathbb{E}[{\mu}_0^{N,h}
        \ast V^N]\right\|_{L^r} \right\|_{L^r(\Omega)}.$$
    Inspecting the computations, we see that full independence of the particles is not mandatory. 
It is sufficient to assume that, for some fixed $k \in \mathbb{N}^*$, and for all $N \ge k$, 
the particles can be partitioned into groups of size $k$ (with possibly a last group of 
smaller size) such that the groups are independent.
Then one can decompose the sum into blocks of size $k$, this yields a sum of about $N/k$ independent centered random variables, and the same moment estimate applies up to a multiplicative constant depending on $k$, 
leading to the same rate in $N$. The computations are 
    similar for the 
    convergence rate of $$\left\| \left\| \tilde{\mu}_0^{N,h} - u_0
    \right\|_{L^1\cap L^r} 
    \right\|_{L^{m}(\Omega)},$$ and we refer the reader to Proposition B.1 in \cite{ORT24}.
    Precisely, under the same 
    assumptions, an explicit convergence rate can be computed, which is 
    negligible compared to \(N^{-\rho + \epsilon}\) and 
    \(N^{\frac{d \alpha}{\bar{r}}} h^{\frac{\zeta}{2}}\).
\end{remark}

%------------------------------------------------------------
%------------------------------------------------------------

\section{Proofs of main results}\label{sec3}
In this section, we provide the detailed proofs for our two main results, 
Theorem \ref{thm.1} and Theorem \ref{thm.2}, along with several 
useful preliminary results.

The main idea behind the proofs of Theorem 
\ref{thm.1} and Theorem \ref{thm.2} is an expansion 
around the Gaussian kernel using Itô's formula. For 
Theorem \ref{thm.1}, we subtract the mild formulation of \( u_t \) from 
that of \( \mu_t^{i,N,h} \). We then decompose the resulting error into 
several terms, bound the stochastic convolution term, apply 
Proposition \ref{prop1.5} to control one of the terms, and use Grönwall's 
lemma to finalize the proof.

For Theorem \ref{thm.2}, we use a 
similar expansion around the Gaussian kernel for 
both \( u_t \) and \( u_t^{i,N,h} \). The error is again decomposed 
into several terms. One term is handled using Theorem \ref{thm.1}, another 
term is analyzed in relation to the error induced by the discretization step 
of the EM scheme, and for both terms, Proposition \ref{prop1.3} below is 
used. Finally, Grönwall's lemma is applied to conclude the proof.

%------------------------------------------------------------
%------------------------------------------------------------

\subsection{Preliminary results}

To begin, we present a preliminary lemma 
that is essential in the calculations for the proofs of
 Proposition \ref{prop1.5} and Theorem \ref{thm.1}.
 This is a general lemma which can be applied to a
  large class of stochastic processes $(X^{i,N})_{1 \leq i \leq N}$ that 
  are not necessarily associated with a specific particle 
  system. The only requirements are that these processes 
  are progressively measurable with respect to the natural 
  filtration of the Brownian motions and that certain
   moment bounds are satisfied.
 The proof of this lemma follows directly 
 from the computations presented in Appendix A.2 of \cite{ORT}.

\begin{lemma}\label{lemma}
    Let $T > 0$. 
    Consider $(W^i)_{i \in \mathbb{N}^*}$ a family of 
    independent Brownian motions. For all $N \in \mathbb{N}^*$ and all
    $i \in \{1, \ldots, N\}$, let $(X_s^{i,N})_{s \in [0,T]}$ be a
    progressively measurable process according to the natural filtration of $W^i$.
    Define, for any $t \in [0,T]$ and $x \in \mathbb{R}^d$, the process
     \begin{equation*}
         M_t^N(x) := \frac{1}{N} \sum_{i=1}^N 
         \int_0^t e^{(t-s)\Delta} \nabla V^N(x-X_s^{i,N}) \cdot dW_s^i.
     \end{equation*}
Then, for any $m \geq 1$ and $p \in [2,\infty]$, there exists 
     $C > 0$ such that for any $N \geq 1$,
     $t \in [0,T]$ and $\varepsilon > 0$,
     \begin{equation*}
         \left\|\sup_{s \in [0,t]} \|M_s^N\|_{L^p} \right\|_{L^m(\Omega)} \leq C N^{-\frac{1}{2}(1-
         2 \alpha d (1-\frac{1}{p}))+\varepsilon}.
     \end{equation*}
 
     Moreover, if for any $q > 0$, 
     \begin{equation*}
        \sup_{N \in \mathbb{N}^*} \sup_{i \in
        \{1,...,N\}} \mathbb{E} \left[ \sup_{s \in [0,T]} |X_s^{i,N}|^q \right] < + \infty,
     \end{equation*}
then, for any $m \geq 1$, $p \in [1,2[$ there exists $C > 0$ such that for any $N \in \mathbb{N}^*$,
     $t \in [0,T]$ and $\varepsilon > 0$,
     \begin{equation*}
         \left\|\sup_{s \in [0,t]} \|M_s^N\|_{L^p}\right\|_{L^m(\Omega)} \leq C N^{-\frac{1}{2}(1-
         \alpha d)+\varepsilon},
     \end{equation*}
\end{lemma}

Next, in Theorem \ref{thm.1}, we require a uniform estimate in 
$N$ for the mollified empirical measure $\tilde{\mu}^{N,h}$.
The following proposition provides such an estimate.

\begin{proposition}\label{prop1.5}
    Under the Assumptions \hyperlink{assumption:A}{\boldmath{$(A_0)$}} and 
    \hyperlink{assumption:Aalpha}{\boldmath{$(A_\alpha)$}}, 
    for any $q \geq 1$, $p \geq 1$ and $T>0$, 
    \[ \sup_{N \in \mathbb{N}} \mathbb{E} 
    \left[ \sup_{t \in [0,T]} \left\| \tilde{\mu}_t^{N,h}
     \right\|_{L^p}^q \right] < + \infty. \]
    \end{proposition}
\begin{proof}
The proof follows similar arguments as the proof of Proposition A.12 
in \cite{ORT} and relies on a mild formulation of the 
mollified empirical measure $\tilde{\mu}^{N,h}$
of the EM discretized system of particles \eqref{eds_euler}.
The mollified empirical measure writes as
$$ \tilde{\mu}^{N,h}_t : x \in \mathbb{R}^d \mapsto 
\frac{1}{N} \sum_{i=1}^N V^N(x-X_t^{i,N,h}).$$

Now fix $x \in \mathbb{R}^d$ and $1 \leq i \leq N$. Applying Itô's 
formula to $s \mapsto g_2(t-s,\cdot) \ast  V^N(x-X_s^{i,N,h})$ for fixed $t$
then summing over $1\leq i \leq N$ leads to
the following equation: 
\begin{align}\label{mild_form_euler}
    \tilde{\mu}_t^{N,h}(x) &= g_2(t,\cdot) \ast \tilde{\mu}_0^{N,h}(x) -
     \frac{1}{N}\sum_{i=1}^N 
    \int_0^t \nabla \cdot g_2(t-s,\cdot) \ast  V^N(x-X_s^{i,N,h})
    F_A(\tilde{\mu}_{\tau_s^h}^{N,h} \ast K(X_{\tau_s^h}^{i,N,h}))
   \,ds \notag \\
    &\quad - \frac{\sqrt{2}}{N}\sum_{i=1}^N 
    \int_0^t g_2(t-s,\cdot) \ast  \nabla V^N(x-X_s^{i,N,h}) \cdot dW_s^{i}.
\end{align}
By triangular inequality, we have $\forall p \geq 1$,
\begin{align*}
    \|\tilde{\mu}_t^{N,h} \|_{L^p} &\leq
    \|g_2(t,\cdot) \ast \tilde{\mu}_0^{N,h}\|_{L^p} \\ 
    &\quad + \left\|\int_0^t
    \nabla \cdot g_2(t-s,\cdot) \ast  \left( \frac{1}{N} \sum_{i=1}^N
    F_A(\tilde{\mu}_{\tau_s^h}^{N,h} \ast K(X_{\tau_s^h}^{i,N,h})
    )V^N(x-X_s^{i,N,h})\right) ds \right\|_{L^p} \\
    &\quad + \left\| \frac{\sqrt{2}}{N} 
    \sum_{i=1}^N \int_0^t g_2(t-s,\cdot) \ast  
    \nabla V^N(x-X_s^{i,N,h}) \cdot 
    dW_s^i \right\|_{L^p}.
\end{align*}
Then using that $F_A$ is bounded and 
inequality \eqref{eq:heat_kernel}, we obtain
\begin{align}\label{Gronw}
    \|\tilde{\mu}_t^{N,h} \|_{L^p} \leq C \left(
    \|\tilde{\mu}_0^{N,h}\|_{L^p} + \int_0^t
    \frac{1}{(t-s)^{1/2}} \|\tilde{\mu}_s^{N,h}\|_{L^p}\,ds+ \sup_{t \in [0,T]} 
    \|M_t^{N,h}\|_{L^p}\right),
\end{align}
where we have set 
$$M_t^{N,h} = \frac{1}{N} \sum_{i=1}^N
    \int_0^t g_2(t-s,\cdot) \ast \nabla V^N(x-X_s^{i,N,h}) 
    \cdot dW_s^i,\quad \forall t \in [0,T].$$
Since $(X_s^{i,N})_{s \in [0,T]}$ is a progressively measurable process
according to the natural filtration of $W^i$
and also a diffusion with bounded coefficients uniformly in $N$, 
hence verifying for any $q>0$,
\[\sup_{N \in \mathbb{N}^*} \sup_{i \in
\{1,...,N\}} \mathbb{E}[\sup_{t \in [0,T]} 
|X_t^{i,N,h}|^q] < + \infty,\] we
obtain directly from Lemma \ref{lemma}
 the following 
bound for any $\varepsilon 
> 0$, any $p \geq 1$, any $ m \geq 1$, there exists $ C > 0$,
\begin{equation}\label{martingale-bounds}
 \left\|\sup_{t \in [0,T]} \left\| 
M_t^{N,h} \right\|_{L^p} \right\|_{L^m(\Omega)}
 \leq 
C N^{-\frac{1}{2}(1-\alpha(d+\chi_p))+\varepsilon},
\end{equation}
where $\chi_p = \max(0,d(1-\frac{2}{p}))$. Using 
Assumption \hyperlink{assumption:Aalpha}{\boldmath{$(A_\alpha)$}},
we obtain a decreasing bound in $N$ in \eqref{martingale-bounds}.
Grönwall's lemma in \eqref{Gronw} and the derived moments of
 $\|M^{N,h}\|_{T,L^p}$ directly provide
 the desired uniform bound in $N$.
\end{proof}

Finally, the proof of Theorem 
\ref{thm.2} relies on the following proposition.

\begin{proposition} \label{prop1.3}
    Let $T>0$, $N\in \mathbb{N}^{\ast}$ and $h>0$. Consider 
    the previously defined EM scheme \eqref{eds_euler} 
    and assume that $\forall i \in \{1,...,N\},  
    \mathcal{L}(X_0^{i,N,h}) = u_0$.
    Then  for any $i \in \{1,...,N\}$, for any $t \in [0,T]$,
     $X_t^{i,N,h}$ admits 
    a density 
    $u_t^{i,N,h}$ such that for
    almost all $ x \in \mathbb{R}^d$,
    \begin{equation}\label{mild_em}
        u_t^{i,N,h}(x)=
        \int_{\mathbb{R}^d}
         g_2(t,x-y) u_0(y) \, dy-
        \int_{0}^t\mathbb{E}[F_A(\tilde{\mu}_{\tau_s^h}^{N,h}\ast 
        K(X_{\tau_s^h}^{i,N,h}))\cdot\nabla g_2(t-s,x-X_s^{i,N,h})]\,ds.
    \end{equation}
    Moreover, for any $c>2$, there exists
     $C > 0 $ independent of $h$ and $N$,
    such that for
    $t \in (0,T]$ and almost all $x\in \mathbb{R}^d$,

    \begin{equation*}\label{density_gaussian_bound}
        u_t^{i,N,h}(x) \le C 
         \int_{\mathbb{R}^d}
        g_c(t,x-y)u_0(y)\,dy.
    \end{equation*}
   
\end{proposition}

\begin{proof}
    
The proof gets its inspiration
from the proof of Proposition
2.1 in \cite{jourdain}. Fix a time horizon $T>0$, a discretization 
step $h >0$
and consider  the EM scheme \eqref{eds_euler}. Let $i \in \{1,...,N\}$. 
In order to obtain \eqref{mild_em}, we apply Itô's formula to \( v(s, X_s^{i,N,h}) \),
where 
\[
v(s, y) = \mathbf{1}_{\{s < t\}} g_2(t-s, \cdot) \ast \phi(y) + \mathbf{1}_{\{s = t\}} \phi(y),
\] 
with \( \phi \)  a \( C^2 \) function with compact support. Hence, we obtain
\begin{equation}\label{eq:ito}
\phi(X_t^{i,N,h}) = 
v(0,X_0^{i,N,h} ) + \int_{0}^{t} \nabla v(s, X_s^{i,N,h}) 
\cdot dW_s^i + \int_{0}^{t} \nabla v(s, X_s^{i,N,h}) 
\cdot F_A(\tilde{\mu}_{\tau_s^h}^{N,h}\ast 
K(X_{\tau_s^h}^{i,N,h}))\,ds.
\end{equation}
Now take the expectation in \eqref{eq:ito}.
Since \(\nabla v\) and \(F_A\) are bounded, it remains to show that  
\(
\mathbb{E} \left[ |\nabla g_2(t - s, X_{s} - y)| \right]
\) 
is bounded by an integrable function in order to apply Fubini’s Theorem
and obtain
\begin{align}
    \int_{\mathbb{R}^d} \phi(y) u_t^{i,N,h}(y) \, dy &= \int_{\mathbb{R}^d} \phi(y) 
    \int_{\mathbb{R}^d} 
        g_2(t,y-x) u_0(x) \, dx \, dy \nonumber \\
        &+ \int_{\mathbb{R}^d} \phi(y) 
        \int_0^t \mathbb{E}[F_A(\tilde{\mu}_{\tau_s^h}^{N,h} \ast K (X_{\tau_s^h}^{i,N,h})) 
        \cdot \nabla g_2(t-s, X_s^{i,N,h} - y)] \, ds. \label{fubini}
    \end{align}
Then since \(\phi\) is arbitrary and \(g_2\) is even in its spatial variable, 
we deduce that $\forall t \in (0,T]$, for almost all $y \in \mathbb{R}^d$,
\begin{equation}\label{EMDuhamel} 
    u_t^{i,N,h}(y)= \int_{\mathbb{R}^d} 
    g_2(t,y-x)u_0(x)\,dx -
     \int_0^t \mathbb{E}[F_A(\tilde{\mu}_{\tau_s^h}^{N,h} \ast K (X_{\tau_s^h}^{i,N,h}))
      \cdot \nabla_y g_2(t-s,y-X_s^{i,N,h})]\, ds.
\end{equation}

The following part is dedicated to the proof of the bound on \(
\mathbb{E} \left[ |\nabla g_2(t - s, X_{s} - y)| \right]
\) necessary for Fubini's Theorem in \eqref{fubini}.

Fix $c > 1$. Let us prove by induction on $k \in \{0,...,\floor{\frac{T}{h}}\}$ that 
for any $t \in (t_k,\max(t_{k+1},T)]$, where $t_k:= kh$,
\begin{equation}\label{bound}
    u_t^{i,N,h}(y) \le 
    c^{\frac{d}{2}\lceil \frac{t}{h} \rceil} \exp\left(\frac{
        \|F_A\|^2_{L^{\infty}}h \lceil 
        \frac{t}{h} \rceil }{4(c-1)}\right)
     \int_{\mathbb{R}^d}
    g_{2c}(t,y-x)u_0(x)\,dx.
\end{equation} 
For $k=0$, we have that for $t \in [0,h]$, 
$X_t^{i,N,h}= X_0^{i,N,h} + F_A(K \ast \tilde{\mu}_{0}^N(X_{0}^{i,N,h}))t + \sqrt{2} W_t^{i}$, hence
\[ u_t^{i,N,h}(y)= \mathbb{E}\big[g_2\big(t, y - X_0^{i,N,h} - 
F_A(K \ast \tilde{\mu}_{0}^N(X_{0}^{i,N,h}))t\big)\big]. \] 
Using the inequality: $\forall c> 1, \forall 
x,y,z \in \mathbb{R}^d,$\color{black}
$$|z-x-y|^2 = |z-x|^2 + |y|^2 - 2(z-x)y
\geq \frac{1}{c}|z-x|^2 - 
\frac{1}{c-1}|y|^2,$$
(the lower bound being a direct consequence of Young's inequality $2a \cdot b \leq
\varepsilon |a|^2 + \frac{1}{\varepsilon} |b|^2$ with $\varepsilon =
1-\frac{1}{c}$) and knowing that $F_A$ is bounded, we have $\forall c > 1$, 
\begin{align*}
 u_t^{i,N,h}(y)&\leq c^{\frac{d}{2}} \E \left[
 \exp\left(\frac{\left|F_A(K \ast \tilde{\mu}_{0}^N(X_{0}^{i,N,h})) t\right|^2}{4(c-1)t}\right) 
 g_{2c}(t, y - X_0^{i,N,h})\right] \\
 &\leq c^{\frac{d}{2}}
  \exp\left(\frac{\|F_A\|^2_{L^{\infty}}h}{4(c-1)}\right)
 \mathbb{E}\big[g_{2c}\big(t, y - X_0^{i,N,h}\big)\big].
\end{align*} \color{black} 
Now take $ k \in \{1,...,\floor{\frac{T}{h}}\}$ and assume that the result holds for $k-1$. 
Then, we have $\forall t \in (t_k,\max(t_{k+1},T)]$, 
$X_t^{i,N,h}= X_{t_k}^{i,N,h} + F_A(K \ast \tilde{\mu}_{t_k}^N(X_{t_k}^{i,N,h}))(t-t_k)
 + \sqrt{2} (W_t^{i}-W_{t_k}^{i})$, 
hence 
\begin{align*}
    u_t^{i,N,h}(y) 
    & = \mathbb{E}\big[g_2\big(t-t_k, y - X_{t_k}^{i,N,h} -
    F_A(K \ast \tilde{\mu}_{t_k}^N(X_{t_k}^{i,N,h}))(t-t_k)\big)\big]\\
    &\leq c^{\frac{d}{2}} \exp\left(\frac{\|F_A\|_{L^{\infty}}^2h}{4(c-1)}\right)
    \mathbb{E}\big[g_{2c}\big(t-t_k, y - X_{t_k}^{i,N,h}\big)\big]\\
    & = c^{\frac{d}{2}} \exp\left(\frac{
        \|F_A\|^2_{L^{\infty}}h}{4(c-1)}\right)
     \int_{\mathbb{R}^d} g_{2c}\big(t-t_k, y - x\big) u_{t_k}(x) \, dx \\
    & \leq c^{\frac{d}{2}\lceil \frac{t}{h} \rceil} \exp\left(\frac{
        \|F_A\|^2_{L^{\infty}}h\lceil 
    \frac{t}{h} \rceil}{4(c-1)}\right)
    \int_{\mathbb{R}^d} g_{2c}\big(t-t_k, y - x\big) 
    \int_{\mathbb{R}^{d}} g_{2c}\big(t_k, x - z\big) u_0(z) \, dz \, dx \\
    & = c^{\frac{d}{2}\lceil \frac{t}{h} \rceil} \exp\left(\frac{
        \|F_A\|^2_{L^{\infty}}h\lceil 
    \frac{t}{h} \rceil}{4(c-1)}\right)
     \int_{\mathbb{R}^d} g_{2c}\big(t, y - x\big) u_0(x) \, dx.
\end{align*} 
This allows to obtain \eqref{bound}, which was a necessary bound in order 
to obtain \eqref{EMDuhamel}.\\
Finally, let
$c>1$,  using the Duhamel formula \eqref{EMDuhamel} and setting $f(t) : = \sup_{y \in \mathbb{R}^d}
\frac{u_t^{i,N,h}(y)}{\int_{\mathbb{R}^d}
g_{2c}(t,y-x)u_0(x)\,dx}$, we have
\begin{align*}
    &\frac{u_t^{i,N,h}(y)}{\int_{\mathbb{R}^d}
g_{2c}(t,y-x)u_0(x)\,dx} 
    \leq C + C \int_0^t \frac{1}{(t-s)^{1/2}} \frac{1}{\int_{\mathbb{R}^d} g_{2c}(t,y-x)u_0(x)\,dx}
    \mathbb{E}\left[ g_{2c}(t-s,y-X_s^{i,N,h}) \right] \, ds \\
    & \quad = C + C \int_0^t \frac{1}{(t-s)^{1/2}} \frac{1}{\int_{\mathbb{R}^d} g_{2c}(t,y-x)u_0(x)\,dx}
    \int_{\mathbb{R}^d} g_{2c}(t-s,y-z) u_s^{i,N,h}(z) \, dz \, ds \\
    &\quad  \leq C + C \int_0^t \frac{1}{(t-s)^{1/2}} f(s) 
    \frac{1}{\int_{\mathbb{R}^d} g_{2c}(t,y-x)u_0(x)\,dx} \int_{\mathbb{R}^d} 
    \int_{\mathbb{R}^d} g_c(t-s,y-z)g_{2c}(s,z-x) u_0(x)\, dx\, dz \, ds\\
    & \quad = C + C \int_0^t \frac{1}{(t-s)^{1/2}} f(s) \, ds.
\end{align*}
Hence, $$f(t) \leq C + C \int_0^t \frac{1}{(t-s)^{1/2}} f(s) \, ds.$$
Then, using Grönwall's lemma in the convolution form, see for example Lemma 
7.1.1 in 
\cite{H81},
we obtain for a.e. $y \in \mathbb{R}^d$,
\begin{equation*}
    u_t^{i,N,h}(y) \le C 
     \int_{\mathbb{R}^d}
    g_{2c}(t,y-x)u_0(x)\, dx,
\end{equation*}
where $C$ does not depend on $h$ and on $N$.
\end{proof}

%------------------------------------------------------------
%------------------------------------------------------------

\subsection{Proof of Theorem 
\ref{thm.1}\label{secthm1}}

Consider the mild formulation \eqref{mild_form_euler} of the 
mollified empirical measure $\tilde{\mu}^{N,h}$:
\begin{align}
\tilde{\mu}_t^{N,h}(x) &= g_2(t,\cdot) \ast \tilde{\mu}_0^{N,h}(x) -
     \frac{1}{N}\sum_{i=1}^N 
    \int_0^t \nabla \cdot g_2(t-s,\cdot) \ast  V^N(x-X_s^{i,N,h})
    F_A(\tilde{\mu}_{\tau_s^h}^{N,h} \ast K(X_{\tau_s^h}^{i,N,h}))
   \,ds \notag \\
    &\quad - \frac{\sqrt{2}}{N}\sum_{i=1}^N 
    \int_0^t g_2(t-s,\cdot) \ast  \nabla V^N(x-X_s^{i,N,h}) \cdot dW_s^{i}.
\end{align}
\textcolor{black}{Recalling Assumption 
\hyperlink{assumption:A-cut-off}{\boldmath{$(A_{F_A})$}}
} (i.e. $A \geq 
\|K \ast u\|_{T,L^{\infty}}$), $\forall s \in [0,T], \, 
F_A(K \ast u_s)= K \ast u_s$. Since $u$ is solution of \eqref{PDE} in 
the sense of Definition \ref{def}, it
satisfies 
the following mild formulation:
\begin{equation*}
    u_t = g_2(t,\cdot) \ast u_0 - \int_0^t \nabla \cdot
    g_2(t-s,\cdot) \ast  (u_s F_A(K \ast u_s))ds.
\end{equation*}
\\
We then subtract and add
the term $$\int_0^t \nabla \cdot g_2(t-s,\cdot) \ast  
\langle \mu_s^{N,h}, V^N(x - \cdot) F_A(K \ast \tilde{\mu}_s^{N,h}(x)) 
\rangle ds = \int_0^t \nabla \cdot g_2(t-s,\cdot) \ast  
(\tilde{\mu}_s^{N,h} F_A(K \ast \tilde{\mu}_s^{N,h}))(x) ds$$ in 
equation \eqref{mild_form_euler} and notice that 
\[ \frac{1}{N} \sum_{i=1}^N \int_0^t 
\nabla g_2(t-s,\cdot) \ast 
V^N(x-X_s^{i,N,h}) \cdot F_A(\tilde{\mu}_s^{N,h} 
\ast K(X_s^{i,N,h}))\,ds = \int_0^t  \nabla \cdot g_2(t-s,\cdot) \ast 
\langle \mu_s^{N,h}, V^N(x-\cdot) 
 F_A(\tilde{\mu}_s^{N,h} \ast K(\cdot)) \rangle \,ds .\]
Hence we obtain the following decomposition
\begin{align*}
    \tilde{\mu}_t^{N,h}(x) - u_t(x) &= 
    g_2(t,\cdot) \ast (\tilde{\mu}_0^{N,h}(x) - u_0(x))
    + \int_0^t \nabla \cdot g_2(t-s,\cdot) \ast 
     ( u_s F_A(K \ast u_s)
        -\tilde{\mu}_s^{N,h} F_A(K \ast \tilde{\mu}_s^{N,h}))(x)ds \\
    &+ E_t^{N,h}(x) - \sqrt{2} M_t^{N,h}(x) + H_t^{N,h}(x),
\end{align*}
where we denote
\begin{equation*}
    E_t^{N,h}(x) = \int_0^t \nabla \cdot g_2(t-s,\cdot) \ast  
    \langle \mu_s^{N,h}, V^N(x - \cdot) \left(F_A(K \ast \tilde{\mu}_s^{N,h}(x))
    - F_A(K \ast \tilde{\mu}_s^{N,h}(\cdot))\right) \rangle ds,
\end{equation*}

\begin{equation*}
    M_t^{N,h}(x) = \frac{1}{N}
    \sum_{i=1}^N \int_0^t g_2(t-s,\cdot) \ast  \nabla
    V^N(x-X_s^{i,N,h}) \cdot dW_s^{i},
\end{equation*}
\begin{equation*}
    H_t^{N,h}(x) = \frac{1}{N} \sum_{i=1}^N \int_0^t
    \nabla \cdot g_2(t-s,\cdot) \ast  V^N(x-X_s^{i,N,h})
    (F_A(K \ast \tilde{\mu}_s^{N,h}(X_s^{i,N,h}))
        -F_A(K \ast \tilde{\mu}_{\tau_s^h}^{N,h}(X_{\tau_s^h}^{i,N,h})))ds.
\end{equation*}
The main difference with the proof in \cite{ORT} is the 
presence of the term $H_t^{N,h}$ which represents the error induced by the 
discretization step of the EM scheme.

By the triangular inequality, we have
\begin{equation}\label{triang}
    \begin{aligned}
        &\| (\tilde{\mu}_s^{N,h} F_A(K \ast \tilde{\mu}_s^{N,h}) - u_s 
    F_A(K \ast u_s)) \|_{L^1 \cap L^r} \\
        &\leq \| u_s (F_A(K \ast \tilde{\mu}_s^{N,h}) - F_A(K \ast u_s)) 
    \|_{L^1 \cap L^r} + \| (\tilde{\mu}_s^{N,h} - 
    u_s) F_A(K \ast \tilde{\mu}_s^{N,h}) \|_{L^1 \cap L^r}.
    \end{aligned}
\end{equation}
Using \eqref{eq:heat_kernel},\eqref{triang}, the Lipschitz property of 
$F_A$ on one term, the boundness of $F_A$ on the other,
the fact that $\|K \ast f\|_{L^{\infty}} \leq 
C \|f\|_{L^1 \cap L^r}$  and that $u_t \in 
L^1 \cap L^r, \forall t \in [0,T]$ (see Definition \ref{def}), we obtain
\begin{align*}
    &\bigg\| \int_0^t \nabla \cdot g_2(t-s,\cdot) \ast 
     (\tilde{\mu}_s^{N,h} F_A(K \ast \tilde{\mu}_s^{N,h})
    - u_s F_A(K \ast u_s))\,ds \bigg\|_{L^1 \cap L^r} \\
    & \leq C \int_0^t \frac{1}{(t-s)^{1/2}}
    \| u_s (F_A(K \ast \tilde{\mu}_s^{N,h}) - F_A(K \ast u_s)) 
    \|_{L^1 \cap L^r} \,ds
     + C \int_0^t \frac{1}{(t-s)^{1/2}}  \| (\tilde{\mu}_s^{N,h} - 
    u_s) F_A(K \ast \tilde{\mu}_s^{N,h}) \|_{L^1 \cap L^r}
    \,ds
     \\
    & \leq C \int_0^t \frac{1}{(t-s)^{1/2}}
    \| u_s \|_{L^1 \cap L^r} \|F_A\|_{Lip}
    \|K \ast(\tilde{\mu}_s^{N,h} - u_s)\|_{L^{\infty}} 
    \, ds
    + C \int_0^t \frac{1}{(t-s)^{1/2}}
    \| F_A \|_{L^{\infty}}
    \| \tilde{\mu}_s^{N,h} - u_s \|_{L^1 \cap L^r} \, ds\\
    &\leq C \int_0^t \frac{1}{(t-s)^{1/2}}
    \| \tilde{\mu}_s^{N,h} - u_s \|_{L^1 \cap L^r} \, ds.
\end{align*}
Hence, we have
\begin{align*}
    \|\tilde{\mu}_t^{N,h} - u_t\|_{L^1 \cap L^r}&\leq
    \|g_2(t,\cdot) \ast (\tilde{\mu}_0^{N,h} - u_0)\|_{L^1 \cap L^r}
    + C \int_0^t \frac{1}{(t-s)^{1/2}} 
    \|\tilde{\mu}_s^{N,h} - u_s\|_{L^1 \cap L^r} \, ds \notag \\
    &\quad + \|E_t^{N,h}\|_{L^1 \cap L^r}
    + \sqrt{2}\|M_t^{N,h}\|_{L^1 \cap L^r} 
    + \|H_t^{N,h}\|_{L^1 \cap L^r}.
\end{align*}
Finally, using Grönwall's lemma for convolution 
integrals and taking the supremum over time, we establish
\begin{equation}\label{different_errors}
    \|\tilde{\mu}^{N,h} - u\|_{t,L^1 \cap L^r} \leq
    C \left( \|\tilde{\mu}_0^{N,h} - u_0\|_{L^1 \cap L^r} 
    + \|E^{N,h}\|_{t,L^1 \cap L^r}
    + \|M^{N,h}\|_{t,L^1 \cap L^r} 
    + \|H^{N,h}\|_{t,L^1 \cap L^r} \right).
\end{equation}

Now we will bound the moments of the different terms which appear in 
\eqref{different_errors}.\\

\textbf{Step 1}: Moments of $\|E^{N,h}\|_{t,L^1 
\cap L^r}$ 

Using the property \eqref{eq:heat_kernel} of the heat kernel, 
we obtain for any $t \in [0,T]$, for any $q \geq 1$, there exists $C > 0$ such that
\begin{equation*}
    \|E_t^{N,h}\|_{L^q} \leq C \int_0^t
    \frac{1}{(t-s)^{1/2}} 
   \left(\int_{\mathbb{R}^d} 
    | \langle \tilde{\mu}_s^{N,h}, V^N(x - \cdot)
    \left(F_A(K \ast \tilde{\mu}_s^{N,h}(x)) - 
    F_A(K \ast \tilde{\mu}_s^{N,h}(\cdot))\right) \rangle
    |^q dx\right) ^{\frac{1}{q}}ds.
\end{equation*}

Then we proceed exactly as in p.14 of \cite{ORT} 
\color{black}to get, 
from the Lipschitz continuity
of $F_A$, \hyperlink{assumption:AK}{$(A^{K}_{iii})$}
and the scaling property $|z|^\zeta V^N(z) \leq N^{-\alpha \zeta} V^N(z)$, the estimate,
taking $q \in \{1,r\}$,
\begin{equation*}
\begin{aligned}
    \left\|\left\|E_t^{N,h}\right\|_{L^1 \cap L^r}\right\|_{L^m(\Omega)}
    &\leq C \left\|
        \int_0^t
        \frac{\left\|\tilde{\mu}_s^{N,h}\right\|_{L^1 \cap L^r}}
             {(t-s)^{1/2}}
        \left(
            \int
            \left|
                \left\langle
                    \tilde{\mu}_s^{N,h},
                    V^N(x-\cdot)\lvert \cdot-x\rvert^\zeta
                \right\rangle
            \right|^q dx
        \right)^{\frac{1}{q}}
        ds
    \right\|_{L^m(\Omega)}
    \\
    &\leq \frac{C}{N^{\alpha\zeta}}
    \left\|
        \int_0^t
        \frac{
            \left\|\tilde{\mu}_s^{N,h}\right\|_{L^1 \cap L^r}
            \left\|\tilde{\mu}_s^{N,h}\right\|_{L^q}
        }{(t-s)^{1/2}}
        ds
    \right\|_{L^m(\Omega)}
    \\
    &\leq \frac{C}{N^{\alpha\zeta}}
    \left(
        \int_0^t
        \frac{
            \mathbb{E}\left[
                \left\|\tilde{\mu}_s^{N,h}\right\|_{L^1 \cap L^r}^{2m}
            \right]
        }{(t-s)^{1/2}}
        ds
    \right)^{\frac{1}{m}}
    \leq \frac{C}{N^{\alpha\zeta}}.
\end{aligned}
\end{equation*}
\color{black}
\textbf{Step 2}: Moments of $\|M^{N,h}\|_{t,L^1 
\cap L^r}$ 

The moments of this quantity follow from Lemma \ref{lemma} and
were obtained 
in \eqref{martingale-bounds} where it is shown that for any $ \varepsilon 
> 0$, for any  $ m \geq 1$, there exists $C > 0$, 

\[ \left\|\sup_{t \in [0,T]} \left\| 
M_t^{N,h} \right\|_{L^1 
\cap L^r} \right\|_{L^m(\Omega)}
 \leq 
C N^{-\frac{1}{2}(1-\alpha(d+\chi_r))+\varepsilon},\]
where we recall that $\chi_r = \max(0,d(1-\frac{2}{r}))$.\\

\textbf{Step 3}: Moments of $\|H^{N,h}\|_{t,L^1 
\cap L^r}$ 

Using the $\zeta$-Hölder property of $K$ given by $(A_{iii}^K)$ and the
Lipschitz property of $F_A$,
\begin{align*}
    &\|H_t^{N,h}\|_{L^1 
    \cap L^r} \\
    &\leq
    \frac{1}{N} \sum_{i=1}^N \int_0^t
    \frac{1}{(t-s)^{1/2}} \|V^N(\cdot - X_s^{i,N,h})
    [F_A(K \ast \tilde{\mu}_{\tau_s^h}^{N,h}(X_{\tau_s^h}^{i,N,h}))
    - F_A(K \ast \tilde{\mu}_s^{N,h}(X_s^{i,N,h}))]\|_{L^1 
    \cap L^r}ds 
    \notag\\
    &\leq \|V^N\|_{L^1 
    \cap L^r} \|F_A\|_{Lip}
    \int_0^t \frac{1}{(t-s)^{1/2}} \frac{1}{N} \sum_{i=1}^N
    |K \ast \tilde{\mu}_{\tau_s^h}^{N,h}(X_{\tau_s^h}^{i,N,h})
    - K \ast \tilde{\mu}_s^{N,h}(X_s^{i,N,h})| ds \notag \\
    & \leq C \|V^N\|_{L^1 
    \cap L^r}
    \int_0^t \frac{1}{(t-s)^{1/2}} \frac{1}{N^2} \sum_{i,j=1}^N
    |K \ast V^N(X_{\tau_s^h}^{i,N,h} - X_{\tau_s^h}^{j,N,h})
    - K \ast V^N(X_s^{i,N,h} - X_s^{j,N,h})| ds \notag \\
     & \leq C \|V^N\|_{L^1 \cap L^r}^2
     \int_0^t \frac{1}{(t-s)^{1/2}} \frac{1}{N^2} \sum_{i,j=1}^N
    | X_{\tau_s^h}^{i,N,h} - X_{\tau_s^h}^{j,N,h} - 
    X_s^{i,N,h} + X_s^{j,N,h} |^{\zeta} ds.
\end{align*}
Since 
\[|X_{s}^{i,N,h}-X_{\tau_s^h}^{i,N,h}| \leq 
|\int_{\tau_s^h}^{s} F_A(\tilde{\mu}_{\tau_r^h}^{N,h} \ast 
K(X_{\tau_r^h}^{i,N,h}))dr|
+ \sqrt{2} |W_{s}^{i}-W_{\tau_s^h}^{i}| \leq C h + 
 \sqrt{2} |W_{s}^{i}-W_{\tau_s^h}^{i}|,\]
and 
$\|V^N\|_{L^p} \leq C N^{\frac{d \alpha}{\bar{p}}}$, for any $p \geq 1$, we obtain 
\begin{align*}
    \|H^{N,h}\|_{T,L^1 
    \cap L^r} &\leq 
    C N^{\frac{2d \alpha}{\bar{r}}} \frac{1}{N^2} \sum_{i,j=1}^N
    \sup_{t \in [0,T]} 
    \int_0^t \frac{1}{(t-s)^{1/2}} 
    (h + \sqrt{2} |W^i_s - W^i_{\tau_s^h}| + 
    \sqrt{2} |W^j_s - W^j_{\tau_s^h}|)
    ^{\zeta} ds \\
    &\leq  C N^{\frac{2d \alpha}{\bar{r}}} \frac{1}{N^2} \sum_{i,j=1}^N
    \sup_{t \in [0,T]} (h + \sqrt{2} |W^i_t - W^i_{\tau_t^h}| +
    \sqrt{2} |W^j_t - W^j_{\tau_t^h}|)^{\zeta}.
\end{align*}
Hence, for any $m \geq 1 $, we have
\begin{align*}
    \left\|\|H^{N,h}\|_{t,L^1 \cap L^r} \right\|_{L^m(\Omega)}
    &\leq C N^{\frac{2d \alpha}{\bar{r}}} \frac{1}{N^2} \sum_{i,j=1}^N
    \left( \mathbb{E} \left[ 
        (h + \sup_{t \in [0,T]} \sqrt{2} |W^i_t - Wi_{\tau_t^h}| +
        \sqrt{2} |W^j_t - W^j_{\tau_t^h}|)^{\zeta m}
    \right] \right)^{\frac{1}{m}},
\end{align*}
then using the BDG's inequality for the
 martingales $(W^i_t - W^i_{\tau_t^h})_{t \in [0,T]}$ for any $i \in \{1,\ldots,N\}$, 
    we obtain for any  $
    m \geq 1/\zeta$,
\begin{equation*}
    \|\|H^{N,h}\|_{t,L^1 \cap L^r} \|_{L^m(\Omega)}
    \leq C N^{\frac{2d \alpha}{\bar{r}}} h^{\frac{\zeta}{2}}.
\end{equation*}

\textbf{Step 4}: Conclusion

Plugging the previous inequalities in Equation
\eqref{different_errors}, 
we conclude that for any $\varepsilon >0$ and $m \geq 1$, 
there exists $C>0$ such that for any $N \in \mathbb{N}^{\ast}$ and 
$h >0$, 

\begin{equation*}
   \left\| \|\tilde{\mu}^{N,h} - u\|_{t,L^1 \cap 
   L^r} \right\|_{L^m(\Omega)} \leq 
   C \left( \left\| \|\tilde{\mu}_0^{N,h} - u_0\|_{L^1 \cap L^r} \right\|_{L^m(\Omega)}
   + N^{-\rho + \epsilon}  + 
   N^{\frac{2d \alpha}{ \bar{r}}} h^{\frac{\zeta}{2}} \right),
\end{equation*}
where $\rho = \min(\alpha \zeta, \frac{1}{2}(1-\alpha(d+\chi_r)))$.

%---------------------------------------------------------------------
%----------------------------------------------------------------------

\subsection{Proof of Theorem 
\ref{thm.2}}\label{secthm2}
\color{black}
In the context of the Euler-Maruyama scheme, the
idea of expressing the density of a singular SDE using the 
Gaussian density is inspired by [21] (we
may refer the reader to \cite{KM17} in the case of SDEs with bounded and Hölder or piece-wise regular
coefficients). 
\color{black}
It has been proved in \cite{ORT} that,
under the previous assumptions,
the non-discretized version of the system 
\eqref{eds_euler} propagates chaos towards the McKean-Vlasov SDE 
\eqref{SDE}.

In particular, when $u_0 \in L^1 \cap L^r$,
 this McKean-Vlasov equation admits a 
unique weak solution up to any $T \leq T_{\max}$ (see \cite{ORT}).
In fact, the solution is also strong, as there is uniqueness in law 
and the linear equation has a bounded drift $K \ast u_t$ and 
as such admits an unique strong solution (see \cite{V81}).
Moreover, its density $u$ is the mild solution of 
Equation \eqref{PDE} in the sense of 
Definition \ref{def} and for $A$ such that $F_A(K \ast u) = K \ast u$, 
 $u$ satisfies the following equation
for all $ t \in [0,T]$ and almost all $y \in \mathbb{R}^d$, 

\begin{equation}\label{dev1}
     u_t(y) = 
    \int_{\mathbb{R}^d}
g_2(t,y-x) u_0(x)\,dx - \int_0^t  \mathbb{E}[F_A(u_{s} 
\ast K (X_{s})) \cdot \nabla 
g_2(t-s,y-X_s)] ds.
\end{equation}

Consider the particles 
$(X^{i,N,h})_{1 \leq i \leq N}$ defined as the strong solutions 
of \eqref{eds_euler}
driven by independent Brownian motions $(W^i)_{1 \leq i \leq N}$ and with 
initial conditions $(X_0^{i,N})_{1 \leq i \leq N}$ of law $u_0$. For each
$i \in \{1,\ldots,N\}$, define $X^i$ as the strong solution of \eqref{SDE}, 
driven by the Brownian motion $W^i$ and with initial
condition $X_0^{i,N}$.
Fix $i \in \{1,\ldots,N\}$, using Proposition \ref{prop1.3}, $X_t^{i,N,h}$ admits,
at time 
$t \in [0,T]$, a density with respect to the Lebesgue measure on $\mathbb{R}^d$ 
 denoted by $y 
\mapsto u_t^{i,N,h}(y)$ and for almost all $y \in \mathbb{R}^d$,
\begin{equation}\label{dev2}
     u_t^{i,N,h}(y)= \int_{\mathbb{R}^d}
    g_2(t,y-x) u_0(x)\,dx - \int_0^t \mathbb{E}[F_A(\tilde{\mu}^{N,h}_{\tau_s^h} 
    \ast K (X_{\tau_s^h}^{i,N,h})) \cdot \nabla 
    g_2(t-s,y-X_s^{i,N,h})]\,ds.
\end{equation}
By taking the difference of \eqref{dev1} and \eqref{dev2}
and introducing several terms to decompose the error, we obtain 
\begin{equation*}
     u_t(y)-u_t^{i,N,h}(y) = E_t^1 + E_t^2 + E_t^3,
\end{equation*}

where \\

$E_t^1 = \int_0^t \mathbb{E}
\left[(F_A(K \ast \tilde{\mu}^{N,h}_{\tau_s^h} (X_{\tau_s^h}^{i,N,h}))
- F_A(K \ast \tilde{\mu}_s^{N,h}(X_s^{i,N,h}))) \cdot 
\nabla g_2(t-s,y-X_s^{i,N,h})\right] ds,$\\

$E_t^2 = \int_0^t \mathbb{E}
\left[(F_A(K \ast \tilde{\mu}_s^{N,h}(X_s^{i,N,h}))
- F_A(K \ast u_s(X_s^{i,N,h}))) \cdot 
\nabla g_2(t-s,y-X_s^{i,N,h})\right] ds,$\\

$E_t^3  = \int_0^t \mathbb{E}
\left[F_A(K \ast u_s(X_s^{i,N,h})) \cdot 
\nabla_y g_2(t-s,y-X_s^{i,N,h})
- F_A(K \ast u_s(X_s^{i})) \cdot 
\nabla g_2(t-s,y-X_s^{i})\right] ds.$\\

The first term provides an 
error due to the discretization in the Euler scheme. 
The next term is handled using the quantitative convergence of 
$\tilde{\mu}^{N,h}$ towards $u$ in an appropriate norm. The last term will
allow us to conclude
using Grönwall's lemma. \\

\textbf{Step 1}: Bound for $E_t^1$

Using the Lipschitz property of $F_A$, the triangular inequality and 
$\zeta$-Hölder property of $K\ast f$, we have:
\begin{align*}
|E_t^1| &\leq \int_0^t \mathbb{E}
\left[ |F_A(K \ast \tilde{\mu}^{N,h}_{\tau_s^h} (X_{\tau_s^h}^{i,N,h}))
- F_A(K \ast \tilde{\mu}_s^{N,h}(X_s^{i,N,h}))| \, | 
\nabla_y g_2(t-s,y-X_s^{i,N,h})| \right] ds\\
&\leq  \int_0^t \mathbb{E}
\left[ |K \ast \tilde{\mu}^{N,h}_{\tau_s^h} (X_{\tau_s^h}^{i,N,h})
- K \ast \tilde{\mu}_s^{N,h}(X_s^{i,N,h})| \, | 
\nabla_y g_2(t-s,y-X_s^{i,N,h})| \right] ds \\
&\leq  \frac{1}{N} \sum_{k=1}^N \int_0^t \mathbb{E}
\left[ |K \ast V^N (X_{\tau_s^h}^{i,N,h}-X_{\tau_s^h}^{k,N,h})
- K \ast V^{N}(X_s^{i,N,h}-X_{s}^{k,N,h})| | 
\nabla_y g_2(t-s,y-X_s^{i,N,h})| \right] ds \\
&\leq \|V^N\|_{L^1 \cap L^r}\frac{1}{N} \sum_{k=1}^N \int_0^t \mathbb{E}
\left[ |X_{\tau_s^h}^{i,N,h}-X_{\tau_s^h}^{k,N,h}
- X_s^{i,N,h}+X_{\tau_s^h}^{k,N,h}|^{\zeta} | 
\nabla_y g_2(t-s,y-X_s^{i,N,h})| \right] ds.
\end{align*}
Using $\|V^N \|_{L^r} = N^{\alpha d/ \bar{r}}\|V \|_{L^r}$ and 
the property \eqref{eq:gaussian_density} of the Gaussian density,
 we have that
 for any $c>2$, there exists $C >0$, 
\begin{align*}
    |E_t^1| &\leq  C N^{\frac{d \alpha}{\bar{r}}} \frac{1}{N} \sum_{k=1}^N
    \int_0^t \frac{1}{\sqrt{t-s}}
     \mathbb{E}[(h+\sqrt{2}|W_{s}^{i}-W_{\tau_s^h}^{i}| + 
     \sqrt{2}  |W_{s}^{k}-W_{\tau_s^h}^{k}|)^{\zeta}
    g_c(t-s,y-X_s^{i,N,h})] ds.
\end{align*}
Hence, using Hölder's inequality in space with $\bar{p} > d$,
 the previous inequality gives 
\begin{align*}
    &|E_t^1|\\
     &\leq C N^{\frac{d \alpha}{\bar{r}}} \frac{1}{N} \sum_{k=1}^N
    \int_0^t \frac{1}{\sqrt{t-s}}
     \mathbb{E}[(h+\sqrt{2}|W_{s}^{i}-W_{\tau_s^h}^{i}| + 
     \sqrt{2}  |W_{s}^{k}-W_{\tau_s^h}^{k}|)^{\zeta
     \bar{p}}]^{1/\bar{p}} 
     \mathbb{E} \left[ g_c(t-s,y-X_s^{i,N,h})^p \right]^{1/p} ds.
\end{align*}
Choosing $c > p$ and applying Proposition \ref{prop1.3}
with $c/p$ we obtain
\begin{align*}
    \mathbb{E} \left[ g_c(t-s,y-X_s^{i,N,h})^p \right]^{1/p}
    &= \left( \int_{\mathbb{R}^d} g_c(t-s,y-z)^p 
    u_s^{i,N,h}(z)\, dz \right)^{1/p} \\
    &\leq C 
    \left( \int_{\mathbb{R}^d}
    \int_{\mathbb{R}^d} 
     g_c(t-s,y-z)^p g_{c/p}(s,z-x)u_0(x)\,dx \, dz \right)^{1/p} \\
    &\leq C \frac{1}{(t-s)^{\frac{d}{2\bar{p}}}}
    \left(\int_{\mathbb{R}^d} 
     g_{c/p}(t,y-x)u_0(x)\,dx\right)^{1/p}.
\end{align*}
Similarly to the proof 
of bounds for 
moments of $\|H_t^{N,h}\|_{t,L^1 \cap L^r}$
in Theorem \ref{thm.1}, we use 
the BDG's inequality for the
 martingales $(W^i_t - W^i_{\tau_t^h})_{t \in [0,T]}$ for any $i \in \{1,\ldots,N\}$, 
and obtain 
\begin{align*}
    |E_t^1|
     &\leq C N^{\frac{d \alpha}{\bar{r}}} h^{\frac{\zeta}{2}}
     \left(\int_{\mathbb{R}^d}
     g_{c/p}(t,y-x)u_0(x)\,dx\right)^{1/p} \int_0^t \frac{1}{(t-s)^{1/2+d/2\bar{p}}}\,ds \\
     & \leq C N^{\frac{d \alpha}{\bar{r}}} h^{\frac{\zeta}{2}}
     \left(\int_{\mathbb{R}^d} 
     g_{c/p}(t,y-x)u_0(x)\,dx\right)^{1/p}.
\end{align*}

\textbf{Step 2}: Bound for $E_t^2$

Using the fact that $F_A$ is Lipschitz, Hölder's inequality
 with $\bar{p} > d$, Theorem \ref{thm.1} and Proposition
 \ref{prop1.5}, we obtain for all $c > 2$,
\begin{align*}
    |E_t^2| &= \left|\int_0^t \mathbb{E}
    \left[(F_A(K\ast \tilde{\mu}_s^{N,h}(X_s^{i,N,h}))
    -F_A(K\ast u_s(X_s^{i,N,h})))\cdot 
    \nabla_y g_2(t-s,y-X_s^{i,N,h})\right]\,ds\right|\\
    & \leq C \int_0^t \frac{1}{(t-s)^{1/2}} \mathbb{E}
    \left[\|\tilde{\mu}^{N,h}-u\|_{T, L^1 \cap L^r }
     g_c(t-s,y-X_s^{i,N,h})\right]\,ds\\
     & \leq C \left(\left\| \left\| \tilde{\mu}_0^{N,h} - u_0
     \right\|_{L^1\cap L^r} 
     \right\|_{L^{\bar{p}}(\Omega)}+N^{-\rho+\epsilon} + 
     N^{\frac{d \alpha}{\bar{r}}} h^{\frac{\zeta}{2}}\right)
     \int_0^t \frac{1}{(t-s)^{1/2}} \mathbb{E}[g_c(t-s,y-X_s^{i,N,h})^p]^{1/p}\,ds\\
     & \leq C \left(\left\| \left\| \tilde{\mu}_0^{N,h} - u_0
     \right\|_{L^1\cap L^r} 
     \right\|_{L^{\bar{p}}(\Omega)}+N^{-\rho+\epsilon} + 
     N^{\frac{d \alpha}{\bar{r}}} h^{\frac{\zeta}{2}}\right) 
     \left(\int_{\mathbb{R}^d} g_{c/p}(t,y-x)u_0(x)\,dx\right)^{1/p}.
\end{align*}

\textbf{Step 3}: Bound for $E_t^3$ and Grönwall's Lemma

Using that $F_A$ is bounded and \eqref{eq:gaussian_density}, we have
for all $c > 2$,
\begin{align*}
    |E_t^3|  &=\left|\int_0^t \mathbb{E}
    \left[F_A(K\ast u_s(X_s^{i,N,h}))\cdot 
    \nabla_yg_2(t-s,y-X_s^{i,N,h})
    -F_A(K\ast u_s(X_s^{i}))\cdot 
    \nabla_yg_2(t-s,y-X_s^{i})\right]\,ds\right|\\
    &= \left|\int_0^t \int_{\mathbb{R}^d}
    \left[F_A(K\ast u_s(z))\cdot 
    \nabla_yg_2(t-s,y-z)u_s^{i,N,h}(z)
    -F_A(K\ast u_s(z))\cdot 
    \nabla_yg_2(t-s,y-z)u_s(z)\right]\,dz\,ds\right|\ \\
    & \leq C
     \int_0^t \int_{\mathbb{R}^d}
    \frac{1}{(t-s)^{1/2}}
    g_c(t-s,y-z) 
    |u_s^{i,N,h}(z)-u_s(z)|\,dz\,ds. 
\end{align*}
Let $$f(s) : = \sup_{z \in \mathbb{R}^d} \frac{
    |u_s^{i,N,h}(z)-u_s(z)|}{(\int_{\mathbb{R}^d}
    g_{c/p}(t,z-x) u_0(x)\,dx)^{1/p}}.$$
Then, combining all the previous inequalities, we obtain
\[ f(t) \leq C (N^{-\rho+\epsilon} + N^{\frac{d
 \alpha}{\bar{r}}}h^{\frac{1}{2}})
+ C \int_0^t  \frac{1}{(t-s)^{1/2+d/2\bar{p}}} f(s) \, ds. \]
Thus, using Grönwall's lemma in the 
convolution form, for almost every
 $z \in \mathbb{R}^d$, for any $t \in [0,T]$:
$$ |u_t^{i,N,h}(z) - u_t(z)|
\leq C \left(\left\| \left\| \tilde{\mu}_0^{N,h} - u_0
\right\|_{L^1\cap L^r} 
\right\|_{L^{\bar{p}}(\Omega)}+N^{-\rho+\epsilon} +
N^{\frac{d \alpha}{\bar{r}}}h^{\frac{\zeta}{2}}\right)
 \left(\int_{\mathbb{R}^d}  g_{c/p}(t,x-z) u_0(x)\,dx\right)^{1/p}.$$

%---------------------------------------------------------------------
%---------------------------------------------------------------------

\section{Examples}\label{sec4} 

In this section, we provide explicit convergence 
rates, \color{black}applying Theorems \ref{thm.1} and \ref{thm.2},
for various theoretically and practically relevant interaction kernels 
satisfying \color{black}
Assumption \hyperlink{assumption:AK}{\boldmath{$(A_K)$}}.
Recall that, for regular enough initial data (see Remark \ref{initial-data}),
the convergence rate \color{black}in Theorems \ref{thm.1} and \ref{thm.2}
\color{black}is of order
\[
O(N^{-v_1} + N^{v_2}h^{v_3}),
\]
where \(v_1, v_2, v_3 > 0\) are defined as follows:
\[
v_1 = \min\left(\alpha \zeta, \frac{1}{2}(1 - \alpha(d + \chi_r))\right) + \varepsilon, \quad
v_2 = \frac{d \alpha}{\bar{r}}, \quad
v_3 = \frac{\zeta}{2},
\]
with \(\chi_r = \max(0, d(1 - 2/r))\) and where parameters \(p, q, r, \zeta, \alpha\) are specified in 
Assumptions \hyperlink{assumption:AK}{\boldmath{$(A_K)$}} and 
\hyperlink{assumption:Aalpha}{\boldmath{$(A_\alpha)$}}.

When considering convergence rates of this form, one can notice that 
best rate is given when one maximizes \(v_1\) and 
\(v_3\), while minimizing \(v_2\), as functions of the parameters \(p, q, r, \zeta, \alpha\). However, 
due to the interdependence of $v_1, v_2, v_3$, achieving these objectives simultaneously is generally 
impossible without decoupling them. 

Assuming there are no computational limitations, we can 
select \(h\) arbitrarily small. In this case, the optimal rate is 
achieved when the two terms \(N^{-v_1}\) and \(N^{v_2}h^{v_3}\) are of the same order, 
which corresponds to choosing \(h = N^{-\frac{(v_1+v_2)}{v_3}}.\)
Substituting this into the convergence rate, the optimal rate simplifies to
 \[O(N^{-v_1}),\] 
where \(v_1\) is maximized as a function of the parameters \(p, q, r, \zeta, \alpha\). 

In practical scenarios, computational complexity plays a critical role.
 An additional consideration for the optimal choice of parameters 
 is to fix the convergence error $N^{-v_1}$ at a given order
 \(\varepsilon > 0\) and determine the parameters that minimize the computational cost of the scheme.
 Considering the physical dimensional $d \in \{1,2,3\}$,
 the total number of operations 
 is of order \( O\left(\frac{N^2}{h}\right) \). 
 This estimate arises from the fact that computing the 
 position of a single particle requires \( O(N) \) 
 operations, there are \( N \) particles whose positions
  must be updated at each step, and the total number of 
  steps is \( O\left(\frac{1}{h}\right) \). Considering that
terms \(N^{-v_1}\) and \(N^{v_2}h^{v_3}\) are of the same order, we obtain the computational cost
we seek to minimize 
\[
O(\varepsilon^{-\frac{2}{v_1} - \frac{1}{v_3} - \frac{v_2}{v_1 v_3}})
\]
as a function of \(p, q, r, \zeta, \alpha\), under the constraints imposed by Assumptions 
\hyperlink{assumption:AK}{\boldmath{$(A_K)$}} and \hyperlink{assumption:Aalpha}{\boldmath{$(A_\alpha)$}}.

These two perspectives highlight the interplay between the 
convergence rate and computational complexity, guiding the choice of parameters in the 
following examples for optimal performance.

\subsection{Regular kernels} For bounded and
Lipschitz continuous 
kernels, the parameters 
of Assumptions \hyperlink{assumption:AK}{\boldmath{$(A_K)$}} and 
\hyperlink{assumption:Aalpha}{\boldmath{$(A_\alpha)$}} are given by:  
\[
p \in (1, \infty], \quad q = \infty, \quad \zeta = 1, \quad 
\alpha \in \left(0, \frac{1}{d}\right).
\]
For the Heaviside kernel 
$K(\cdot)x := \mathbf{1}_{\{\cdot \geq 0\}}$, we derive the following 
parameters of
Assumptions \hyperlink{assumption:AK}{\boldmath{$(A_K)$}} and 
\hyperlink{assumption:Aalpha}{\boldmath{$(A_\alpha)$}}:
\[
p \in [1, \infty], \quad q = \infty, 
\quad \zeta = 1^-, \quad \alpha \in \left(0, \frac{1}{d} \right).
\]
For this type of kernel, we have \(v_3 = \frac{1}{2}\), and the maximum of \(v_1\) and the 
minimum of \(v_2\) are attained for the same parameters:  
\[
r = 1, \quad \alpha = \frac{1}{d+2}.
\]
In this case, the convergence rate is of order:  
\[
N^{-\frac{1}{d+2}^-} + h^{\frac{1}{2}}.
\]
\color{black}
Let us compare this rate with the one obtained in \cite{BT95} for
one-dimensional nonlinear PDEs with bounded Lipschitz drift, and briefly
discuss the rates established in \cite{BT95,BT94} for one-dimensional PDEs
of Burgers type. For the latter, direct comparisons are more delicate, but
are nevertheless worth highlighting. For the former, neglecting the error
at \(t=0\), Theorem~2.5 gives, for \(d=1\), the rate
\(
    \mathcal{O}\bigl(N^{-1/3}+h^{1/2}\bigr),
\)
under slightly weaker working assumptions than those of Theorem~2.2
in \cite{BT95}, where the error is
\(
    \mathcal{O}\bigl(N^{-1/3}+h^{1/3}\bigr)
\)
(obtained by minimizing therein the fixed
smoothness parameter \(\varepsilon\) into
\(
    \varepsilon
    \approx
    \bigl(N^{-1/2}+h^{1/2}\bigr)^{1/3}
\)).
For the one-dimensional Burgers equation, instead of a direct
interpretation of the PDE (which would involve the kernel
\(K=\delta_{\{0\}}\), see e.g. \cite{ORT}), the authors introduce
a clever alternative in which the solution of the PDE is defined by the
cumulative distribution function
\(
    \int_{-\infty}^{x} u_t(y)\,dy,
\)
with equation~(1.2) subject to the kernel
\(
    K(x)=\mathbf{1}_{\{x\geq 0\}}
\)
(see Section~1 of \cite{BT95} for further details on this interpretation).
The corresponding stochastic particle scheme can then be defined without
mollification, and the associated error can be measured between
\(
    \int_{-\infty}^{x} u_t(y)\,dy
    \quad\text{and its empirical counterpart}\quad
    \frac{1}{N}\sum_{i=1}^{N}
    \mathbf{1}_{\{X_t^{i,N,h}\leq x\}}.
\)
In \cite{BT95,BT94}, this error is defined as the expected
\(L^1\)-distance between the cumulative distribution functions, equivalent to
\(
    \mathbb{E}\!\left[
        \mathbb{W}_1\bigl(\widetilde{\mu}_t^{N,h},u_t\bigr)
    \right],
\)
where \(\mathbb{W}_1\) denotes the
Kantorovich--Rubinstein-Wasserstein distance and is shown to be of order
\(
    \mathcal{O}\bigl(N^{-1/2}+h^{1/2}\bigr)
\)
(see Theorem~3.1 in \cite{BT94} and Theorem~2.1 in \cite{BT95}).
In comparison, observe that by the
 Kantorovich duality, for every probability 
measure
\(\mu\) having a finite first moment, verifies
\[
\begin{aligned}
    \mathbb{W}_1\bigl(\mu,\mu*V^N\bigr)
    &=
    \sup_{\|f\|_{\mathrm{Lip}}\leq 1}
    \int
    \bigl((f*V^N)(x)-f(x)\bigr)\,\mu(dx) \\
    &\leq
    N^{-\alpha}
    \int (1+|x|)\,\mu(dx).
\end{aligned}
\]
Hence, applying Theorem~2.5 again with \(d=1\) and \(\alpha=1/3\), it
follows that
\[
\begin{aligned}
    \mathbb{E}\!\left[
        \mathbb{W}_1\bigl(\mu_t^{N,h},u_t\bigr)
    \right]
    &\leq
    \mathbb{E}\!\left[
        \mathbb{W}_1
        \bigl(\mu_t^{N,h},\widetilde{\mu}_t^{N,h}\bigr)
    \right]
    +
    \mathbb{E}\!\left[
        \mathbb{W}_1
        \bigl(\widetilde{\mu}_t^{N,h},u_t\bigr)
    \right] \\
    &=
    \mathcal{O}\bigl(N^{-1/3}\bigr)
    +
    \mathcal{O}\bigl(N^{-1/3^-}+h^{1/2}\bigr).
\end{aligned}
\]
Compared with \cite{BT95,BT94}, this estimate preserves the rate with
respect to the time step \(h\), but the rate with respect to the number of
particles \(N\) is weaker, passing from \(N^{-1/2}\) to \(N^{-1/3^-}\).
This deterioration is precisely due to the regularization sequence \(V^N\),
which imposes the scaling \(N^{-\alpha}\). This scale is inherent in our
particle method and to our error analysis. In \cite{BT95,BT94}, the authors
exploit the absence of mollification and the regularizing effect of the
Heaviside kernel to obtain a more straightforward error analysis, avoiding
the cost of mollification and the control of fluctuations (appearing in
Step~3 of the proof of Theorem~2.5). Our general method directly addresses
these more intricate components. Although our rate does not directly cover
the particular setting of \cite{BT95,BT94}, the method remains robust for
more singular kernels (see below).
\color{black}
\subsection{Riesz potentials with $\mathbf{d \geq 2}$ and 
$ 
\mathbf{s \in (0,d-2)}$} The general definition of Riesz 
potentials in any dimension is given, up to some $d$-dependent factor,
by 

\begin{equation}\label{eq:Riesz}
    V_s(x) =
    \begin{cases} 
        |x|^{-s} & \text{if } s \in (0,d) \\
        -\log |x|  & \text{if } s = 0
    \end{cases}, \,
    x \in \mathbb{{R}}^d,
\end{equation}
\color{black}
see for example \cite{S24}. \color{black}
Let $$K_s := \pm \nabla V_s,$$ be the corresponding kernel. 
Then, for $d \geq 2$ and 
$s \in (0, d - 2)$, the kernel $K_s$ satisfies  Assumptions 
\hyperlink{assumption:AK}{\boldmath{$(A_K)$}}, see Section 
5 in \cite{ORT}.

Precisely, for $s \in (0,d-2)$ and $d \geq 2$, we have 
$\nabla V_s = -s  \frac{x}{|x|^{s+2}}$. In this case, Assumption 
\hyperlink{assumption:AK}{\boldmath{$(A_K)$}} is satisfied for $p < \frac{d}{s+1}$ and $q > \frac{d}{s+1}$. 
The best convergence rate when considering that 
$N^{-v_1}$ and $N^{v_2}h^{v_3}$ are of the same order is
$$N^{-\frac{1}{2(d+1)}^-} + N^{\frac{d}{2(d+1)}^+} h^{\frac{1}{2}^-},$$ where 
we have set $r=z=+\infty$, $\zeta = 1^-$, $\alpha = \frac{1}{2(d+1)}^+ $.

If on the contrary we want to optimize the computational cost,
we obtain for an error of order $\varepsilon$ the computational cost of order
$$O(\varepsilon^{-(6d+5)^+}),$$
attained for the same parameters as above : $r=z=+\infty$, $\zeta = 1^-$, $\alpha = \frac{1}{2(d+1)}^+ $.
Hence the two perspectives are equivalent in this case.

\subsection{Coulomb type kernels with $\mathbf{d \geq 2}$
 and $\mathbf{s = d-2}$}
An important class of kernels that enter in the above framework are
Coulomb type kernels. They include for 
example kernels derived from the Riesz potential 
\eqref{eq:Riesz} for all $d \geq 2$ and $s = d-2$ : 
$$K = \pm \nabla V_{d-2},$$
 but also the attractive Keller-Segel kernel in $d = 2$ defined by 
 $$K_{KS}(x) = - \chi \frac{x}{|x|^d},$$
 where $\chi$ is a fixed parameter.
For such kernels, we obtain the same convergence rates as above.
In $d = 2$, this gives a rate of order 
$$ N^{-\frac{1}{6}^-} + N^{\frac{1}{3}^+} h^{\frac{1}{2}^-},$$
attained for $r=z=+\infty$, $\zeta = 1^-$, $\alpha = \frac{1}{6}^+ $.
The computational cost for a fixed error $\varepsilon$ is of order
$$O(\varepsilon^{-11^+}).$$

\subsection{Krylov-Röckner type kernels}
In Remark \ref{remlplq}, we considered
$L^p(\mathbb{R}^d)$ kernels which are of the form: 
\[
   K(x) = \frac{x}{|x|^{\gamma}} \chi(x),
\] 
where $\gamma \in (1, 2)$ and $\chi$ is a smooth 
function equal to $1$ on $B_1$ and $0$ outside $B_2$. In this case, 
$K$ verifies the
Krylov-Röckner condition and also 
verifies Assumption \hyperlink{assumption:AK}{\boldmath{$(A_K)$}}.
For such kernel, 
we have 
$p \in (1,\frac{d}{\gamma -1})$, $q = + \infty$,
and we can verify 
that the optimal rate is of order $$N^{-\frac{1}{2(d+1)}^-} + N^{\frac{d}{2(d+1)}^+} h^{\frac{1}{2}^-},$$ where 
we have set $r=+\infty$, $\zeta = 1^-$, $\alpha = \frac{1}{2(d+1)}^+ $.
 And the computational cost for a fixed error $\varepsilon$ is of order
$$O(\varepsilon^{-(6d+5)^+}),$$
attained for the same parameters.\\

\color{black}
For general $K \in L^p$ kernels
satisfying the condition $p > d$, the
Hölder regularity assumption \hyperlink{assumption:AK}{$(A^{K}_{iii})$}
 may not hold.
Yet, our proof strategy of Theorems 2.5 and 2.6
 can still be applied up to some additional
condition on $\mu_0^{N,h}$, $\alpha$ and a weaker convergence rate.
\color{black}
In this case, we can work with the $L^{\bar p}$ norm 
instead of the $L^1 \cap L^r$ norm. Notice that 
Proposition 1.2 in \cite{ORT}, which 
ensures existence and uniqueness of a mild solution 
to the PDE \eqref{PDE}, holds for $L^p$ kernels with $p > d$ 
and for $u_0 \in L^1 \cap L^r$ with $r>\bar{p}$.

Inspecting our previous proofs, we 
notice that \hyperlink{assumption:AK}{$(A^{K}_{iii})$}
is only used to control
the moments of $\|E^{N,h}\|_{t,L^1 \cap L^r}$ and 
the moments of $\|H^{N,h}\|_{t,L^1 \cap L^r}$ in the 
proof of Theorem~\ref{thm.1}, and to control the term $E_t^1$ in the
proof of Theorem~\ref{thm.2} \color{black}
(see respectively, Steps 1 and 3, and Step 1).
\color{black}
Without the Hölder regularity 
of $K * f$, these terms must 
be estimated by exploiting more carefully the 
regularity of $\tilde{\mu}^{N,h}$.

\color{black}
For the proof of Theorem 2.5, let us
re-examine each of the aforementioned controls. \color{black}For the term
\(\|E^{N,h}\|_{t,L^{\bar{p}}}\), we have \begin{align*}  
 \|E_t^{N,h}\|_{L^{\bar{p}}} &\leq
  \int_0^t \|\nabla g_2(t-s,\cdot)\|_{L^{\bar{p}}}
    \left\|\int_{\R^d} d\mu_s^{N,h}(y) V^N(\cdot - y) 
    \left(F_A(K \ast \tilde{\mu}_s^{N,h}(\cdot))
    - F_A(K \ast \tilde{\mu}_s^{N,h}(y))\right)\right\|_{L^1} ds\\
    & \leq C \int_0^t ds \, (t-s)^{-\frac{1}{2}-\frac{d}{2p}}
    \|K \ast \tilde{\mu}_s^{N,h}\|_{W^1_{\infty}} \int_{\R^d} 
    \int_{\R^d}
     V^N(x - y) |x-y|\, d\mu_s^{N,h}(y)\, dx\\
    & \leq C \|\tilde{\mu}_s^{N,h}\|_{t,W^1_{\bar{p}}} 
    N^{-\alpha} \int_0^t (t-s)^{-\frac{1}{2}-\frac{d}{2p}}\, ds,
\end{align*}
where $\|\cdot\|_{t,W^1_{\bar{p}}}: = 
\sup_{s \in [0,t]} \|\cdot\|_{W^1_{\bar{p}}}$, for $W^1_{\bar{p}}$ 
the Sobolev space. Hence, we obtain for all 
$m \geq 1$,
\[
\left\| \|E_t^{N,h}\|_{L^{\bar{p}}} \right\|_{L^m(\Omega)} \leq
  C \left\|\|\tilde{\mu}_s^{N,h}\|_{t,W^1_{\bar{p}}} \right\|_{L^m(\Omega)}
    N^{-\alpha}.
\]
\color{black}Recall that $W_p^1 = H^1_p$ (e.g. \cite{T83}, p.38) where, 
for $\beta  \in \mathbb{N}^*$ and $q\in (1,\infty)$, $H^\beta_q=
\{f \in \mathcal{C}_c^{\infty} : (I + \Delta)^{\beta/2} f \in L^q\}$
is the Bessel potential space. \color{black} 
By revisiting the proof of Proposition A.12 in \cite{ORT}, 
\color{black}
we can verify that for $m\geq1$,
 under the assumption $\left\|\|\tilde{\mu}_0^{N,h}\|_{
 H^r_q}
 \right\|_{L^m(\Omega)} < \infty$, and imposing the restrictive
  version of $(A_\alpha)$ :
  $$ 0 < \alpha < \frac{1}{d+2+2d(\frac{1}{2}-
  \frac{1}{\bar{p}}) },$$
we can show that
$\left\|\sup_{s \in [0,T]} \|\tilde{\mu}_s^{N,h}\|_{H^r_q}
 \right\|_{L^m(\Omega)}$
 is uniformly bounded in $N$ and $h$. 
 \color{black}
 Hence, we obtain a convergence rate of order 
  $N^{-\alpha}$ for the term $E_t^{N,h}$.
  Repeating the estimates for $H_t^{N,h}$ gives
  \begin{align*}
    &\|H_t^{N,h}\|_{L^{\bar{p}}} \\
    &\leq
    \frac{1}{N} \sum_{i=1}^N \int_0^t \| \nabla g_2(t-s,\cdot)\|_{L^{\bar{p}}}
    \|V^N(\cdot - X_s^{i,N,h})
    [F_A(K \ast \tilde{\mu}_{\tau_s^h}^{N,h}(X_{\tau_s^h}^{i,N,h}))
    - F_A(K \ast \tilde{\mu}_s^{N,h}(X_s^{i,N,h}))]\|_{L^1}ds 
    \notag\\
    &\leq C \|V^N\|_{L^1} \|F_A\|_{Lip}
    \int_0^t \frac{1}{(t-s)^{\frac{1}{2}+\frac{d}{2\bar{p}}}} 
    \frac{1}{N} \sum_{i=1}^N
    |K \ast \tilde{\mu}_{\tau_s^h}^{N,h}(X_{\tau_s^h}^{i,N,h})
    - K \ast \tilde{\mu}_s^{N,h}(X_s^{i,N,h})| ds \notag \\
    & \leq C 
    \int_0^t \frac{1}{(t-s)^{\frac{1}{2}+\frac{d}{2\bar{p}}}} \frac{1}{N^2} \sum_{i,j=1}^N
    |K \ast V^N(X_{\tau_s^h}^{i,N,h} - X_{\tau_s^h}^{j,N,h})
    - K \ast V^N(X_s^{i,N,h} - X_s^{j,N,h})| ds \notag \\
     & \leq C \|\nabla V^N\|_{L^{\bar{p}}}
     \int_0^t \frac{1}{(t-s)^{1/2}} \frac{1}{N^2} \sum_{i,j=1}^N
    | X_{\tau_s^h}^{i,N,h} - X_{\tau_s^h}^{j,N,h} - 
    X_s^{i,N,h} + X_s^{j,N,h}| ds.
\end{align*}
Then following the same computations as in Step 3 of the proof 
of Theorem \ref{thm.1}, we obtain for any $m \geq 1$,
\[\left\| \|H_t^{N,h}\|_{L^{\bar{p}}} 
\right\|_{L^m(\Omega)} \leq
  C N^{\alpha(\frac{d}{p}+1)}h^{\frac{1}{2}}. \]
Injecting the previous bounds in the proof of Theorem \ref{thm.1}, 
we conclude that for any $\varepsilon >0$ and $m \geq 1$, 
there exists $C>0$ such that for any $N \in \mathbb{N}^{\ast}$ and 
$h >0$, 

\begin{equation}\label{lp-cv-rate}
   \left\| \|\tilde{\mu}^{N,h} - u\|_{t,L^{\bar{p}}
   } \right\|_{L^m(\Omega)} \leq 
   C \left( \left\| \|\tilde{\mu}_0^{N,h} - u_0\|_{L^{\bar{p}}} \right\|_{L^m(\Omega)}
   + N^{-\rho + \epsilon}  + 
 N^{\alpha(\frac{d}{p}+1)}h^{\frac{1}{2}} \right),
\end{equation}
where $\rho = \min(\alpha, \frac{1}{2}(1-\alpha(d+\chi_{r})))$ and 
$r\geq\bar{p}$.

Finally, for the term $E_t^1$, revisiting the proof of 
Step 1 in Theorem \ref{thm.2} gives
\begin{align*}
|E_t^1| &\leq \int_0^t \mathbb{E}
\left[ |F_A(K \ast \tilde{\mu}^{N,h}_{\tau_s^h} (X_{\tau_s^h}^{i,N,h}))
- F_A(K \ast \tilde{\mu}_s^{N,h}(X_s^{i,N,h}))| \, | 
\nabla_y g_2(t-s,y-X_s^{i,N,h})| \right] ds\\
&\leq  \int_0^t \mathbb{E}
\left[ |K \ast \tilde{\mu}^{N,h}_{\tau_s^h} (X_{\tau_s^h}^{i,N,h})
- K \ast \tilde{\mu}_s^{N,h}(X_s^{i,N,h})| \, | 
\nabla_y g_2(t-s,y-X_s^{i,N,h})| \right] ds \\
&\leq  \frac{1}{N} \sum_{k=1}^N \int_0^t \mathbb{E}
\left[ |K \ast V^N (X_{\tau_s^h}^{i,N,h}-X_{\tau_s^h}^{k,N,h})
- K \ast V^{N}(X_s^{i,N,h}-X_{s}^{k,N,h})| | 
\nabla_y g_2(t-s,y-X_s^{i,N,h})| \right] ds \\
&\leq \|\nabla V^N\|_{L^{\bar{p}}}\frac{1}{N} \sum_{k=1}^N \int_0^t \mathbb{E}
\left[ |X_{\tau_s^h}^{i,N,h}-X_{\tau_s^h}^{k,N,h}
- X_s^{i,N,h}+X_{\tau_s^h}^{k,N,h}|^{\zeta} | 
\nabla_y g_2(t-s,y-X_s^{i,N,h})| \right] ds.
\end{align*}
then, replicating the rest of the proof and knowing that 
$\|\nabla V^N\|_{L^{\bar{p}}} \leq C N^{\alpha(\frac{d}{p}+1)}$,
we obtain for almost any $z \in \mathbb{R}^d$, for any $q 
\in (1,\frac{d}{d-1})$,
 for any
 $t \in [0,T]$:
$$ |u_t^{i,N,h}(z) - u_t(z)|
\leq C \left(\left\| \left\| \tilde{\mu}_0^{N,h} - u_0
\right\|_{L^{\bar{p}}} 
\right\|_{L^{\bar{q}}(\Omega)}+ N^{-\rho + \epsilon}  + 
 N^{\alpha(\frac{d}{p}+1)}h^{\frac{1}{2}}\right)
 \left(\int_{\mathbb{R}^d}  g_{c/q}(t,x-z)
  u_0(x)\,dx\right)^{1/q},$$
where $\rho = \min(\alpha, 
\frac{1}{2}(1-\alpha(d+\chi_{r})))$ and $r\geq\bar{p}$.\\

Finally, for these $L^p$ kernels with $p>d$
 we obtain an optimal rate 
of convergence of order 
$$ N^{-\frac{1}{2+d}^-} + 
N^{\frac{1}{2+d}(\frac{d}{p}+1)} h^{\frac{1}{2}}$$
for $r = + \infty$ and $\alpha = 
\frac{1}{2+d}$.\\

\textbf{Acknowledgements.} 
This work was partially
supported by the French National
 Research Agency through the  
 SDAIM project (ANR-22-CE40-0015).
The author would like to 
thank an anonymous referee for detailed 
comments and suggestions, which have helped to
 improve the quality of this paper and 
 provided valuable insights. 

%------------------------------------------------------------
%------------------------------------------------------------

\addcontentsline{toc}{section}{References} \bibliographystyle{plainnat}
\bibliography{manuscript}

\end{document}